\documentclass{amsart}

\usepackage{amsmath,amsthm,mathrsfs}
\usepackage{amssymb,amsfonts}
\usepackage{url,relsize,xcolor}
\usepackage{tikz,tikz-cd}
\usepackage{enumitem}
\usepackage[hidelinks]{hyperref}

\theoremstyle{plain}
\newtheorem{theorem*}{Theorem}
\newtheorem{theorem}{Theorem}

\newtheorem{proposition}[theorem]{Proposition}
\newtheorem{lemma}[theorem]{Lemma}
\newtheorem{corollary}[theorem]{Corollary}
\newtheorem{claim}[theorem]{Claim}

\numberwithin{equation}{section}

\theoremstyle{definition}
\newtheorem{definition}[theorem]{Definition}

\newtheorem{notation}[theorem]{Notation}

\theoremstyle{remark}
\newtheorem{remark}[theorem]{Remark}

\newtheorem{example}[theorem]{Example}

\DeclareMathOperator{\CH}{\mathrm{CH}}

\DeclareMathOperator{\dom}{dom}

\newcommand{\liftphi}{ \Phi_*}
\DeclareMathOperator{\Fin}{Fin}

\newcommand{\pprod}{\rm{prod}}

\newcommand{\even}{\textrm{even}}
\newcommand{\odd}{\textrm{odd}}

\tikzset{
 symbol/.style={
 draw=none,
 every to/.append style={
 edge node={node [sloped, allow upside down, auto=false]{$#1$}}}
 }
}

\numberwithin{theorem}{section}

\newcommand{\bbN}{{\mathbb N}}

\newcommand{\bbZ}{{\mathbb Z}}

\newcommand{\bbI}{{\mathbb I}}

\newcommand{\cF}{{\mathcal F}}
\newcommand{\cG}{{\mathcal G}}
\newcommand{\cV}{{\mathcal V}}
\newcommand{\cP}{{\mathcal P}}
\newcommand{\cM}{{\mathcal M}}
\newcommand{\cN}{{\mathcal N}}
\newcommand{\cX}{{\mathcal X}}
\newcommand{\cY}{{\mathcal Y}}
\newcommand{\cZ}{{\mathcal Z}}
\newcommand{\cB}{{\mathcal B}}
\newcommand{\cC}{{\mathcal C}}
\newcommand{\cA}{{\mathcal A}}

\newcommand{\cI}{{\mathcal I}}
\newcommand{\cJ}{{\mathcal J}}
\newcommand{\cT}{{\mathcal T}}

\newcommand{\<}{\langle}
\renewcommand{\>}{\rangle}
\newcommand{\fc}{\mathfrak{c}}

\newcommand{\rs}{\restriction}
\newcounter{my_enumerate_counter}
\newcommand{\pushcounter}{\setcounter{my_enumerate_counter}{\value{enumi}}}
\newcommand{\popcounter}{\setcounter{enumi}{\value{my_enumerate_counter}}}

\theoremstyle{theorem}
\newtheorem{thm}{Theorem}

\newtheorem{prop}[thm]{Proposition}

\newcommand{\bbP}{\mathbb P}

\newcommand{\calH}{{\mathcal H}}

\newcommand{\calD}{{\mathcal D}}

\newcommand{\calL}{\mathcal L}

\newcommand{\bbG}{\mathbb G}

\newcommand{\cU}{\mathcal U}
\newcommand{\cW}{\mathcal W}

\DeclareMathOperator{\OCA}{{\mathsf {OCA}}}
\DeclareMathOperator{\OCAT}{{\mathsf {OCA}_T}}
\DeclareMathOperator{\OCAinfty}{{\mathsf {OCA}_\infty}}
\DeclareMathOperator{\OCAsharp}{{\mathsf {OCA}^\#}}
\DeclareMathOperator{\PFA}{{\mathsf {PFA}}}
\DeclareMathOperator{\MA}{{\mathsf {MA}_{\aleph_1}}}
\DeclareMathOperator{\MAsl}{{\mathsf {MA}_{\aleph_1}(\sigma\textrm{-linked})}}

\DeclareMathOperator{\ZFC}{{\mathsf {ZFC}}}

\newcommand{\twoN}{\{0,1\}^\bbN}
\newcommand{\twolN}{\{0,1\}^{<\bbN}}

\newcommand{\cstar}{$\mathrm{C^*}$}

\newcommand{\cPN}{\cP(\bbN)}
\newcommand{\cPNF}{\cP(\bbN)/\Fin}
\DeclareMathOperator{\id}{id}
\DeclareMathOperator{\eq}{eq}

\newcommand{\cJprod}{\cJ_{\pprod}}

\newcommand{\prMF}{\prod_n \cM_n/\Fin}
\newcommand{\prNF}{\prod_n \cN_n/\Fin}
\newcommand{\prLF}{\prod_n \calL_n/\Fin}
\newcommand{\prLpF}{\prod_n \calL'_n/\Fin}

\newcommand{\prMI}{\prod_n \cM_n/\cI}
\newcommand{\prNJ}{\prod_n \cN_n/\cJ}

\newcommand{\prM}{\prod_n \cM_n}
\newcommand{\prN}{\prod_n \cN_n}

\newcommand{\X}{{\mathcal X}}
\newcommand{\N}{{\mathbb N}}
\newcommand{\J}{{\mathcal J}}
\newcommand{\Jcont}{\cJ_{\textrm{cont}}}

\DeclareMathOperator{\Diag}{Diag}
\DeclareMathOperator{\Diff}{Diff}

\title{Trivial Isomorphisms between Reduced Products}

\author{Ben De Bondt}
\address[BDB]{
 Institut de Math\'ematiques de Jussieu - Paris Rive Gauche (IMJ-PRG)\\
 Universit\'e Paris Cit\'e\\
 B\^atiment Sophie Germain\\
 8 Place Aur\'elie Nemours \\ 75013 Paris, France}
\curraddr{Universität Münster\\
		Institut für Mathematische Logik und Grundlagenforschung %-		Fachbereich Mathematik und Informatik
		\\ 
		Einsteinstr. 62\\
		48149 Münster,
		Germany}
	\email{bdebondt@uni-muenster.de}

\author{Ilijas Farah}
\address[IF]{Department of Mathematics and Statistics\\
 York University\\
 4700 Keele Street\\
 North York, Ontario\\ Canada, M3J 1P3\\
 and 
 Ma\-te\-ma\-ti\-\v cki Institut SANU\\
 Kneza Mihaila 36\\
 11\,000 Beograd, p.p. 367\\
 Serbia}
\email{ifarah@yorku.ca}
\urladdr{https://ifarah.mathstats.yorku.ca}
\author{Alessandro Vignati}
\address[AV]{
 Institut de Math\'ematiques de Jussieu - Paris Rive Gauche (IMJ-PRG)\\
 Universit\'e Paris Cit\'e, Institut Universitaire de France\\ 
 B\^atiment Sophie Germain\\
 8 Place Aur\'elie Nemours \\ 75013 Paris, France}
\email{ale.vignati@gmail.com}
\urladdr{http://www.automorph.net/avignati}

\date{\today}

\begin{document}
 \maketitle

 \begin{abstract}
 We introduce a general method for showing that under relatively weak forcing axioms all isomorphisms between reduced products of countable models of certain theories are trivial (with a natural definition of `trivial isomorphism'). This applies to theories of countable fields, linear orders, trees, and random graphs. We also show that Todor\v cevi\'c’s Open Colouring Axiom, $\mathsf {OCA}_{\mathrm T}$, implies that all automorphisms of $\mathcal P(\mathbb N)/\mathrm{Fin}$ are trivial (by a 1993 result of Veli\v{c}kovi\'c, this is a consequence of $\mathsf {OCA}_{\mathrm T}$ and $\mathsf{MA}_{\aleph_1}$).
 \end{abstract}
 \setcounter{tocdepth}{1}
\tableofcontents

\section{Introduction}

We study reduced products of countable structures and their isomorphisms. Fix a countable language $\mathcal L$ and countable $\mathcal L$-structures $\cM_n$ and~$\cN_n$. 
Our basic questions are the following. 
\begin{enumerate}
 \item How does the isomorphism type of a reduced product $\prod_n \cM_n/\cI$ depend on the structures $\cM_n$ and on the ideal $\cI$? 
 \item Can one isolate the \emph{right} notion of a trivial isomorphism between reduced products, and prove that forcing axioms imply all isomorphisms are trivial? 
\end{enumerate}
These questions have been lingering around for a while, and both the second and the third author considered them unavoidable, but most likely formidably difficult and requiring substantial technical innovations. It turns out that, at least in some natural categories and in the case when $\cI$ is the Fr\'echet ideal $\Fin$, this was not quite the case. 
Let us consider some instructive examples. 

Reduced products of vector spaces are maximally non-rigid. 
If each $\cM_n$ is a vector space over a field~$F$, then $\prod_n \cM_n/\cI$ is a vector space over $F$. Thus two such reduced products are isomorphic if and only if they have the same dimension, and the structure `forgets’ the ideal $\cI$. If $\cI=\Fin$, or any other analytic ideal, and all $\cM_n$ and $F$ are countable, then this dimension is $2^{\aleph_0}$, and $\prod_n \cM_n/\cI$ has~$2^{2^{\aleph_0}}$ automorphisms. In \cite{debondt2024saturation} we generalized this observation by proving that every reduced product over $\Fin$ whose theory is stable is $\fc$-saturated. Therefore every reduced product of countable structures over $\Fin$ whose theory is stable has $2^{2^{\aleph_0}}$ automorphisms, provably in ZFC. 

If the Continuum Hypothesis, CH, holds, then the property of having $2^{2^{\aleph_0}}$ many automorphisms extends to reduced products $\prod_n \cM_n/\Fin$ for any sequence of structures $\cM_n$ of the same countable language that are of cardinality $\leq\aleph_1$. This is a consequence of the fact that such reduced products are countably saturated (called $\aleph_1$-saturated by some authors), see \cite{ChaKe}. By CH such reduced products have cardinality $\aleph_1$ and are therefore fully saturated. 

The following example is worth taking a closer look at. Let each $\cM_n$ be the two-element Boolean algebra and let each $\cN_n$ be the one-dimensional vector space over the field with two elements $\mathbb{F}_2$. Then the structures $\cM_n$ can be considered as expansions of the structures~$\cN_n$ and for the reduced product $\prod_n \cN_n/\cI$ we have maximal non-rigidity, as pointed out above. The situation for the structure $\prod_n \cM_n/\cI$, which is isomorphic to the quotient Boolean algebra $\cP(\bbN)/\cI$, is more complex. W. Rudin had proven that CH implies $\mathcal P(\bbN)/\Fin$ has $2^{2^{\aleph_0}}$ automorphisms (\cite{Ru}). By today’s standards, this is, as already mentioned in the previous paragraph, an easy consequence of the fact that $\cP(\bbN)/\Fin$ is countably saturated: countable saturation allows one to run a transfinite back-and-forth argument of length $\omega_1$ that builds a complete binary tree of height $\omega_1$, all branches of which give rise to distinct automorphisms of the structure. 

Are there obstructions to constructing a complete binary tree whose branches correspond to automorphisms of $\cP(\bbN)/\Fin$ when CH fails? Hausdorff constructed such an obstruction in the form of an $(\aleph_1,\aleph_1)$-gap in $\cP(\bbN)/\Fin$, thus proving that (in the modern terminology) $\cP(\bbN)/\Fin$ is not $\aleph_2$-saturated, and therefore that if CH fails then some partial automorphism of an $\aleph_1$-sized subalgebra of $\cP(\bbN)/\Fin$ cannot be extended to an automorphism of the entire structure (\cite{Hau:Summen}).\footnote{Hausdorff’s motivation for constructing such gap was different; at the time it appeared to be a partial result towards proving CH in ZFC.} This presents an obstruction to a recursive construction of an automorphism of $\cP(\bbN)/\Fin$ when CH fails.\footnote{It should be noted however that the Boolean algebra $\cP(\bbN)/\Fin$ has $2^{2^{\aleph_0}}$ automorphisms in Cohen’s original model of ZFC in which CH fails, by~\cite{ShSte:Non-trivial}. This direction of research will not be pursued in the present paper.} 

Remarkably, S. Shelah has proven that in some forcing extension of the universe the situation is totally different: all automorphisms of $\cP(\bbN)/\Fin$ are \emph{trivial} (\cite{Sh:PIF}). An automorphism $\Phi$ of $\cP(\bbN)/\cI$ is \emph{trivial} if there are sets $A$ and $B$ whose complements are in $\cI$ and a bijection $f\colon A\to B$ such that $X\mapsto f[X]$ lifts $\Phi$. It is not difficult to see that this is equivalent to asserting that~$\Phi$ has a lifting that is a continuous homomorphism. 
Shelah’s conclusion has since been deduced from the Proper Forcing Axiom (\cite{ShSte:PFA}) and from the conjunction of two of its consequences $\OCAT$ and $\MA$ (\cite{Ve:OCA}). While $\MA$ alone does not suffice for this conclusion by \cite[Theorem~3.1]{Ve:OCA}, until now it was open whether $\OCAT$ alone does. 

For general reduced products and other quotient structures, isolating the correct notion of a trivial automorphism (or trivial isomorphism) is a nontrivial issue. For example, in case of the Calkin algebra, trivial automorphisms coincide with inner automorphisms (\cite{Fa:All}, see also \cite[\S 17]{Fa:STCstar}), but for other corona \cstar-algebras many trivial automorphisms are outer (\cite{vignati2018rigidity}). In case of \v Cech--Stone remainders, the situation is analogous to that of $\cP(\bbN)/\Fin$ (\cite[Theorem~C]{vignati2018rigidity}). In the present paper we will not delve into \cstar-algebras or other metric structures. 

The main focus of this paper is to extend the rigidity results for automorphisms of $\mathcal P(\bbN)/\Fin$ to isomorphisms of reduced products of the form $\prod_n\cM_n/\Fin$, where~$\cM_n$ are countable structures in a given language. Specifically, we substantially enlarge the class of structures whose reduced products satisfy rigidity \`a la Shelah, under suitable Forcing Axioms. This typically proceeds through the following two steps.\footnote{A two-step analysis analogous to the one given below applies to isomorphisms, and in some categories even to homomorphisms, between reduced products.}\footnote{The proof offered in the present paper follows a slightly different strategy.} (Note that since all $\cM_n$ are assumed to be at most countable, $\prod_n \cM_n$ carries a natural Polish space topology.) 
\begin{enumerate}
 \item Use forcing axioms to prove that every isomorphism between reduced products in a given class has a Borel lifting. 
 \item Prove that every isomorphism (between reduced products) with a Borel lifting is trivial. 
\end{enumerate}
The latter step applies in a wider context, including the case of reduced powers of finite groups (\cite{Fa:Liftings}, \cite{KanRe:Ulam}) in which there may be automorphisms without a Borel lifting. (See our earlier remark on vector spaces and reduced products whose theory is stable.) We are however interested in a situation in which `rigidity’ is interpreted in the strongest way possible, where arbitrary isomorphisms (not necessarily with a Baire-measurable lifting) are allowed. While referring the reader to \cite{farah2022corona} for a bigger picture of the rigidity program pursued here we proceed to describe our main results.

Towards a strong version of rigidity for isomorphisms of reduced products, we introduce a sharpening of Todor\v cevi\'c's Open Colouring Axiom, $\OCAT$ (Definition~\ref{Def.OCAsharp}). This axiom, denoted $\OCAsharp$, substantially simplifies the `usual' uniformization arguments needed when proving rigidity results (see e.g. \cite[forcings $\cP$ and $\cP_{\omega_1}$ in \S 3]{Fa:AQ}, \cite{mckenney2018forcing}, or \cite{farah2022corona}). We prove that $\OCAsharp$ alone implies all automorphisms of $\cP(\bbN)/\Fin$ are trivial, and that it suffices even for the case $\cI=\Fin$ of the $\OCA$ lifting theorem of \cite{Fa:AQ}, originally proven by an involved uniformization argument using Martin's Axiom (see Theorem~\ref{T.OCAsharp-Fin}). Even better, we adapt the proof in \cite[\S 5]{moore2021some} to show that $\OCAT$ implies $\OCAsharp$, therefore giving the following strengthening of the main result of \cite{Ve:OCA} and of the $\OCA$ lifting theorem (\cite[Theorem 3.3.5]{Fa:AQ}) for $\Fin$. 

\begin{thm}\label{thm:OCAPomega}
$\OCAT$ implies all automorphisms of $\cP(\bbN)/\Fin$ are trivial. It moreover implies that every homomorphism $\Phi\colon \cP(\N)/\Fin\to \cP(\N)/\Fin$ has a decomposition $\Phi = \Phi_1\oplus\Phi_2$, where $\Phi_1$ has a completely additive lifting and where the kernel of $\Phi_2$ is nonmeager. 
\end{thm}

We now list some surprising corollaries of our results. When two sequences $\cM_n$ and $\cN_n$ of structures in the same finite first order language are given, we say that two reduced products $\prod_n \cM_n/\Fin$ and $\prod_n \cN_n/\Fin$ are \emph{obviously isomorphic} if there is a bijection $f$ between cofinite subsets of $\bbN$ such that $\cM_{f(n)}\cong \cN_n$ for all~$n$. An isomorphism constructed from data of this form is called \emph{trivial}. (To circumvent easy obstructions as in Example~\ref{Ex.trivial}, the general definition of trivial isomorphism in case the language is infinite is slightly more involved, see Definition~\ref{Def.trivial}.) 

\begin{thm} \label{T.Fields} 
 Assume $\OCAT+\MAsl$. 
 Two reduced products of countable fields are isomorphic if and only if they are obviously isomorphic. Moreover, every automorphism of such reduced product, and every isomorphism between such reduced products, is trivial.
\end{thm}

\begin{thm}\label{T.LO} 
 Assume $\OCAT+\MAsl$. 
 Two reduced products of countable linear orderings are isomorphic if and only if they are obviously isomorphic. Moreover, every automorphism of such reduced product, and every isomorphism between such reduced products, is trivial. 
\end{thm}

\begin{thm}\label{T.Trees} 
 Assume $\OCAT+\MAsl$. 
 Two reduced products of countable trees are isomorphic if and only if they are obviously isomorphic. Moreover, every automorphism of such reduced product, and every isomorphism between such reduced products, is trivial. 
\end{thm}

In the next result, the graph $\bbG_{n,p}$ denotes the random graph with~$n$ vertices in which every pair of distinct vertices is connected with probability $p\in (0,1)$ and these Bernoulli trials are independent. Thus $\bbG_{n,p}$ is not a graph but a random variable that ranges over graphs with $n$ vertices, where these vertices are identified with $\{0,\dots, n-1\}$. The space of such graphs is finite, and to any given graph $\cG$ in this space this random variable assigns the probability of $p^k (1-p)^{\binom {n}{2}-k}$, where $k$ is the number of edges of $\cG$.\footnote{The vertices of graph $\cG$ are labelled, and this is not the probability that a random graph is isomorphic to $\cG$.} We define the product of random graphs $\bbG_{\infty,p}$ to be the random variable $\prod_n \bbG_{n,p}$, where $\bbG_{n,p}$ are independent random graphs. 
In other words, for every $0<p<1$ we have a Borel probability measure on the space of all products of graphs $\prod_n \cG_n$ such that $\cG_n$ has vertices $\{0,\dots, n-1\}$; we denote this measure $\mu_p$.  

\begin{thm} \label{T.Graphs} 
Assume $\OCAT+\MAsl$. 
For all fixed $0<p<1$ and $0<q<1$, the set of all pairs $(\bar \cG, \bar \calH)$ of sequences of graphs as in the previous paragraph and such that there is a nontrivial isomorphism between $\prod_n\cG_n/\Fin$ and $\prod_n \mathcal H_n/\Fin$ has $\mu_p\times\mu_q$ measure 0. 
 
Also, all automorphisms of the reduced power of the Rado graph (i.e., the countably infinite random graph) are trivial. 
\end{thm}

The conclusion of Theorem~\ref{T.Graphs} applies to a rather wide class of (what we call) \emph{sufficiently random graphs} (Definition~\ref{Def.SuffRG}). 
This is a fairly weak requirement that applies to random graphs chosen using different laws, such as the sparse random graphs of \cite{shelah1988zero}.

To put our rigidity results in a proper context, we contrast them with some well-known facts; a proof of the following is in \S\ref{S.CH}. (Analogous results for metric structures appear in \cite{ghasemi2014reduced}.) 

\begin{prop} \label{P.CH} Assume the Continuum Hypothesis. 
 \begin{enumerate}
 	\item \label{0.P.CH} If $\calL$ is a countable language and $\cM_n$, for $n\in \bbN$, is a sequence of countable $\calL$-structures, then it has a subsequence such that the reduced products over $\Fin$ of any two further subsequences are isomorphic. 
 	In particular:
 	\begin{enumerate}
 \item \label{1.P.CH} Every sequence of countable fields $\cF_n$, for $n\in \bbN$, has a subsequence such that the reduced products over $\Fin$ of any two further subsequences are isomorphic.
 \item \label{2.P.CH} Every sequence of countable trees $\cT_n$, for $n\in \bbN$, has a subsequence such that the reduced products over $\Fin$ of any two further subsequences are isomorphic. 
 \end{enumerate}
 \item \label{3.P.CH}There is a partial ordering $\cP$ such that for every sequence $\calL_n$, for $n\in \bbN$, of countable discrete linear orderings with endpoints such that for every $m\in\bbN$ the set $\{n\mid |\calL_n|\leq m \}$ is finite,\footnote{We allow for the possibility that some, or even all, of $\calL_n$ are infinite.} we have $\prod_n \calL_n/\Fin\cong \cP$. 
 \item \label{4.P.CH} There is a graph $\cG$ such that for every $p\in (0,1)$ the set of sequences $(\cG_n)_n$ of graphs such that $ \prod_n \cG_{n}/\Fin\cong \cG$ 
 has full $\mu_p$-measure.\footnote{This is the measure associated with the definition of a random graph. See the first paragraph of the proof of Theorem~\ref{T.Graphs}.}
 \end{enumerate}
\end{prop}

One component of our proofs is the introduction of a \emph{coordinate-respecting property} of isomorphisms between reduced products (Definition~\ref{def:CR}), and a related first-order property (Definition~\ref{Def.CR.Th}). In case of reduced products of fields, this reduces (no pun intended) to finding a definable copy of $\cP(\bbN)/\Fin$ in the reduced product. 
In the case of linear orderings and of random graphs, we find many copies of $\cP(\bbN)/\Fin$ in the reduced product definable from parameters, as well as a system of isomorphisms between them, definable from the same parameters.

As a byproduct, our methods also give that isomorphisms between powers of structures in the categories that we consider are trivial in the appropriate sense (see \S\ref{S.products}). 

The following is our main technical result, proven in~\S\ref{S:proof}; its proof filters through our axiom $\OCAsharp$ (although this axiom is a consequence of $\OCAT$, it was instrumental in finding and presenting the proof). 

\begin{thm}\label{th:main}
 Assume $\OCAT+\MAsl$. Then every coordinate-respecting isomorphism between reduced products modulo $\Fin$ of countable structures in the same language is trivial. 
\end{thm}

Proving this, and the fact that the reduced products mentioned in Theorems~\ref{T.Fields}, ~\ref{T.LO}, ~\ref{T.Trees}, and ~\ref{T.Graphs} have indeed only coordinate-respecting isomorphisms (\S\ref{S.CoordinateRespecting}), gives our main results. 

The use of $\MAsl$ in Theorems~\ref{T.Fields}--\ref{T.Graphs} and Theorem~\ref{th:main} is most likely only a matter of convenience. All uses of $\MAsl$ in this paper filter through the proof of Proposition~\ref{lem:IprodccmodFin}. It would be desirable to derive these conclusions from $\OCAT$ alone.

\subsubsection*{Structure of the paper} In \S\ref{S.CoordinateRespecting} we introduce the notions of coordinate-respecting functions and of theories which recognize coordinates. A reduced product $\prod_{n\in \bbN} M_n/\cI$ of models of a theory that recognizes coordinates contains a definable copy of $\cP(\bbN)/\cI$, or at least a workable poor man's version of one.\footnote{The analogous assertion applies to reduced products over larger index sets, but we will not explore consequences of this observation in the present paper.} 
An isomorphism between such reduced products gives an isomorphism between these definable Boolean algebras, and is therefore coordinate-respecting.\footnote{This is not nearly as straightforward as it may sound; see Definition~\ref{def:locallyCR}.} We then prove that theories of fields, linear orders, trees, and certain graphs recognize coordinates. Sections \ref{S.OCAsharp} and \ref{S.Uniformization.Fin} are dedicated to the axiom $\OCAsharp$: after introducing it, we prove that it follows from $\OCAT$ and that it implies all automorphisms of $\cP(\bbN)/\Fin$ are trivial as well as a more general statement, the conclusion of the $\OCA$ lifting theorem (\cite[Theorem 3.3.5]{Fa:AQ}) for $\Fin$. This provides a proof of Theorem~\ref{thm:OCAPomega}. 
 Section \ref{S:proof} is dedicated to the proof of our main technical result, Theorem~\ref{th:main}, and in \S\ref{S.Applications} we prove Theorems~\ref{T.Fields}--\ref{T.Graphs}. In \S\ref{S.CH} we analyse the (lack of) rigidity under CH; these results are straightforward and included only for completeness. Lastly, \S\ref{S.products} is an intriguing footnote to our main results. In it we give situations in which automorphisms (and isomorphisms) of arbitrary (not reduced!) products are products of isomorphisms of the fibers. These results rely on a unique factorization result for products and are, to the best of our knowledge, new. Appendix \ref{S.sigma-Borel} contains the proof of a well-known, but perhaps not so well-documented, descriptive set theoretic fact used in our uniformization arguments in \S\ref{S.Uniformization.Fin} and \S\ref{S:proof}.

\subsection*{Acknowledgements}
This work was completed during a visit of IF to the Institut de Math\'ematiques de Jussieu, in Paris in the winter of 2023. This visit was funded by Universit\'e Paris Cit\'e and I.F.’s NSERC grant, and the authors are grateful for funding. 
During the ultimate stage of revision, BDB was funded by the Deutsche Forschungsgemeinschaft (DFG, German Research Foundation) under Germany's
Excellence Strategy EXC 2044 –390685587, Mathematics Münster: Dynamics–Geometry–Structure. AV is funded by the Institut Universitaire de France. 

We would like to thank Zo\'e Chatzidakis, Mirna D\v zamonja, and Boban Veli\v ckovi\'c for their remarks and Su Gao for raising the question on whether our results can be extended to reduced products of trees. 
We are also grateful to the anonymous referee for providing a very useful report that considerably improved the presentation and pointing out an issue (now fixed) with the statement of Theorem~\ref{T.Graphs}. Finally, we would like to thank Jo\'e Faber for pointing out a mismatch in an earlier version between the statement of Proposition~\ref{P.Jcont.nonmeager} and its use in the proof of Theorem~\ref{T.OCAsharp-Fin}.

\section{Coordinate-respecting functions} \label{S.CoordinateRespecting}
Suppose that $\cM_n$, for $n\in \bbN$, are structures in the same language~$\calL$. For simplicity of notation we will assume that $\calL$ is a single-sorted language. 
Let $\cI$ be an ideal on $\bbN$ that includes the Fr\'echet ideal $\Fin$. Recall that the symbol $\forall^\mathcal I n$ stands for `for all $n$ not in a set belonging to $\mathcal I$', or for `for all $n$ but $\mathcal I$ many'.

The \emph{reduced product} $\prod_n \cM_n/\cI$ is the $\calL$-structure defined as follows. (For simplicity of notation we provide the definition only in case when $\calL$ is a single-sorted language; modifying the definition to fit the multi-sorted case is straightforward.)
For sequences $a=(a(n))$ and $b=(b(n))$ in $\prod_n \cM_n$, let 
\[
a\sim_\cI b\quad\text{ if }\quad(\forall^\mathcal I n)\, a(n)=b(n).
\]
The universe of $\prod_n \cM_n/\cI$ is the set of $\sim_\cI$ equivalence classes. We write $(a(n))/\cI$ for an equivalence class when needed, but often deliberately confuse an equivalence class with a representing sequence\footnote{Just like in the case of $L_p$ spaces, this practice is mostly harmless.}.

For $k\geq 1$, a $k$-ary function symbol $f$ in $\calL$, and a $k$-tuple $\bar a$ in $\prod_n \cM_n/\cI$ let $f(\bar a)$ be the $\sim_\cI$-equivalence class of $(f(\bar a (n)))_n$. 

For $k\geq 1$, a $k$-ary relation symbol $R$ in $\calL$, and a $k$-tuple $\bar a$ in $\prod_n \cM_n/\cI$ let 
 \[
 \prod_n \cM_n/\cI\models R(\bar a)\quad\text{ if }\quad \forall^\mathcal I n\, (\cM_n\models R(\bar a(n))).
 \] 
By $\pi_\cI\colon \prod_n \cM_n\to \prod_n \cM_n/\cI$ we denote the quotient map. 
We will mostly be concerned with the special case $\cI = \Fin$, so we simply write $\pi$ for $\pi_{\Fin}$.
If $a\in\prod_n\cM_n/\cI$ and there is no danger of confusion we write $\tilde a\in \prod_n\cM_n$ for a lift of $a$ (the choice of the lift is irrelevant). Likewise, if $S\in\mathcal P(\bbN)/\cI$, $\tilde S$ is any lift of $S$.
 
We abuse terminology and say that $\cI$ is \emph{atomless} if the Boolean algebra $\cP(\bbN)/\cI$ is. Being atomless is equivalent to the nonexistence of an $\mathcal I$-positive $S\subseteq\cP(\bbN)$ such that $\cI\restriction S=\{A\in\cI\mid A\subseteq S\}$ is a maximal ideal in $\cP(S)$. Note that every atomless ideal automatically includes the Fr\' echet ideal. 

The most straightforward way to construct functions between reduced products is to permute indices and to compose with a product of fiberwise maps. The following definition formalizes this fact.

\begin{definition} \label{Def.trivial}
Let $\cM_n$ and $\cN_n$, for $n\in\bbN$, be sets of cardinality at least two. Let $\cI$ and $\cJ$ be atomless ideals and fix a function $\Phi\colon \prod_n\cM_n/\cI\to\prod_n\cN_n/\cJ$.

We say that:
\begin{enumerate}
\item $\Phi$ is of \emph{product form} if there are $h_n\colon \cM_n\to \cN_n$ such that the function $a\mapsto (h_n(a(n)))_n$ lifts $\Phi$. Such function will be denoted $\Phi[(h_n)]$. 
\item $\Phi$ is of \emph{twisted product form}, if there are $f\colon \bbN\to \bbN$ and $ h_n\colon \cM_{f(n)}\to \cN_{n}$, for $n\in \bbN$, such that the function $a\mapsto ( h_n(a(f(n))))$ lifts $\Phi$: 
\[
\begin{tikzpicture}
\matrix[row sep=1cm,column sep=1cm]
{
& & \node (M1) {$\prod_n \cM_n$}; && &\node (M2) {$\prod_n \cN_n$};&\\
& & \node (Q1) {$\prod_n \cM_n/\cI$}; &&& \node (Q2) {$\prod_n \cN_n/\cJ$} ;\\
 };
 \draw (M1) edge [->] node [above] {$a\mapsto ( h_n(a(f(n))))$} (M2);
 \draw (Q1) edge [->] node [above] {$\Phi[f,( h_n)]$} (Q2);
 \draw (M1) edge [->] node [left] {$\pi_\cI$} (Q1);
 \draw (M2) edge [->] node [left] {$\pi_\cJ$} (Q2);
 \end{tikzpicture}
\]
Such function will be denoted $\Phi[f,( h_n)]$.
\end{enumerate}
Further, suppose that all the $\cM_n$ and $\cN_n$, for $n\in\bbN$, are $\mathcal L$-structures for the same first-order language $\mathcal L$. 
\begin{enumerate}\setcounter{enumi}{2}
\item \label{3.Def.trivial} A homomorphism\footnote{Meaning, a homomorphism of $\mathcal L$-structures} $\Phi\colon \prod_n\cM_n/\cI\to\prod_n\cN_n/\cJ$ is \emph{trivial} if it is of twisted product form and for every finite sublanguage $\calL_0$ of $\calL$, for all but $\cJ$ many~$n$, $h_n$ is a homomorphism between $\calL_0$-reducts of $\cM_n$ and $\cN_n$. 
\end{enumerate}
\end{definition}

If the language $\calL$ is finite then condition \eqref{3.Def.trivial} reduces to the assertion that $h_n$ is a homomorphism for all but $\cJ$ many $n$. In general this simpler condition is strictly stronger, see Example~\ref{Ex.trivial}.

We already observed that the Boolean algebra $\mathcal P( \bbN)/\Fin$ is isomorphic to the reduced power over the Fr\'echet ideal of the two-element Boolean algebra $\{0,1\}$. Since $\{0,1\}$ has no endomorphism other than the identity, our definition of trivial coincides with the usual definition of trivial endomorphism of $\mathcal P(\bbN)/\Fin$.

In Lemma~\ref{L.Phi.1}, and elsewhere, we assume that the universe of every structure has at least two distinct elements. The proof of this lemma is a simple exercise.

\begin{lemma}\label{L.Phi.1}
Let $\cM_n$ and $\cN_n$, for $n\in\bbN$, be sets of cardinality at least two. Let $\cI$ and $\cJ$ be atomless ideals. Suppose that $\Phi\colon \prod_n\cM_n/\cI\to\prod_n\cN_n/\cJ$ is a function of twisted product form, $\Phi=\Phi[f, (h_n)]$. 
\begin{enumerate}
\item \label{1.L.Phi.1} Then $f$ is a Rudin--Keisler reduction: $X\in \cI$ if and only if $f^{-1}[X]\in \cJ$, for all $X\subseteq \bbN$. 
\item \label{3.L.Phi.1} If $\Phi$ is a bijection, then 
\[
\{n\mid h_n\text{ is not a bijection}\}\in \cJ.
\]
\end{enumerate}
Further, suppose that all the $\cM_n$ and $\cN_n$, for $n\in\bbN$, are $\mathcal L$-structures for the same first-order language $\mathcal L$. 
\begin{enumerate}\setcounter{enumi}{2}
\item \label{2.L.Phi.1} If $\Phi\colon \prod_n\cM_n/\cI\to\prod_n\cN_n/\cJ$ is a homomorphism and the signature $ \calL$ is finite, then 
\[
\{n\mid h_n\text{ is not a homomorphism}\}\in \cJ
\]
and $\Phi$ is trivial.
\end{enumerate}
In particular, all homomorphisms between reduced products which are of twisted product form are trivial.\qed
\end{lemma}

\begin{example}\label{Ex.trivial}The conclusion of Lemma~\ref{L.Phi.1} \eqref{2.L.Phi.1} can fail if the finiteness assumption is dropped. More precisely, there are a language $\mathcal L$, $\mathcal L$-structures $\mathcal M_n$, and a trivial automorphism of product form $\Phi\colon \prod_n \cM_n/\Fin\to \prod_n \cM_n/\Fin$ such that it is not possible to choose all but finitely many of $h_n\colon \cM_n\to \cM_n$ to be homomorphisms. Let $\calL$ be the language with infinitely many unary predicates $P_n(x)$, for $n\in \bbN$. For each~$n$ take $\cM_n$ to be equal to the unique countable model $\cM$ in which $P_k^\cM$, for $k\in \bbN$, form an independent family (i.e.\ the set $\bigcap_{k\in K}P_k^\cM\cap\bigcap_{l\in L}M\setminus P_l^\cM$ is nonempty for every two disjoint finite subsets $K,L$ of $\bbN$). For a fixed $n$ the sets 
	\[
	A_F=\bigcap_{k\in F}P_k^\cM\cap\bigcap_{l\in n\setminus F}M\setminus P_l^\cM, \qquad \text {for $F\subseteq n$}
	\] 
	form a partition of $\cM_n$ into $2^n$ infinite pieces. Independence implies that there is a permutation $h_n$ of $\cM_n$ such that $h_n[A_F]=A_F$ for all $F\subseteq n$ but $P_{n+1}(x)\Leftrightarrow \lnot P_{n+1}(h_n(x))$ for all $x\in \cM_n$. Then $\Phi[(h_n)]$ is an automorphism of product form. If $h_n'\colon \cM_n\to \cM_n$ are such that all but finitely many of them are homomorphisms, then $h_n\neq h_n'$ for all but finitely many $n$, hence $\Phi[(h_n')]\neq \Phi[(h_n)]$. 
\end{example}

Functions of twisted product form come together with $f\colon \bbN\to\bbN$ which, as we showed in Lemma~\ref{L.Phi.1}\eqref{1.L.Phi.1}, induces a homomorphism $\mathcal P(\bbN)/\cI\to\cP(\bbN)/\cJ$.
Coor\-di\-nate-respecting homomorphisms between quotients defined below provide an intermediate notion between arbitrary homomorphisms and those of twisted product form.

\begin{notation} \label{Notation.2.5}
Let $\cM_n$, for $n\in\bbN$, be sets, and let $\cI$ be an ideal on $\bbN$. If $S\subseteq \bbN$ we write 
\begin{equation}\label{eq.MrsS}
\left(\prMI\right)\rs S=\prod_{n\in S}\cM_n/\cI, 
\end{equation}
with the convenient convention that if $S\in \cI$ then $\left(\prMI\right)\rs S$ is a one-element set. 
The very same notation is used for $S\in \cPN/\cI$ to denote the quotient $\left(\prMI\right)\rs \tilde S$, for any lift $\tilde S\subseteq \bbN$ of $S$ (this does not depend on the choice of $\tilde S$, and it is a convenience that will not cause confusion). We let 
\[
 \pi_S\colon \prMI\to \left(\prMI\right)\rs S
\]
 be the quotient map and for $a,b$ in $\prMI$ write 
\[
a=_S b\text{ if }\pi_S(a)=\pi_S(b) 
\] 
(with understanding that $S$ may be a nonzero element of $\cPN/\cI$ or an $\cI$-positive subset of $\bbN$, with the distinction both clear from the context and inconsequential). 
\end{notation}

\begin{definition}\label{def:CR} 
Let $\cM_n$ and $\cN_n$, for $n\in\bbN$, be sets of cardinality at least two. Let $\cI$ and $\cJ$ be atomless ideals.
A function $\Phi\colon \prod_n\cM_n/\cI\to\prod_n\cN_n/\cJ$ is called \emph{coordinate-respecting} if there exists a homomorphism 
\[
\alpha=\alpha^\Phi\colon \cPN/\cI\to \cPN/\cJ
\]
such that for all $a,b$ in $\prod_n\cM_n/\cI$ and $S\in \cPN/\cI$ we have that
 \begin{equation} \label{eq.CR}
\text{ if }a=_{S} b\text{ then } \Phi(a)= _{\alpha(S)}\Phi(b). 
\end{equation}
A function $\Phi\colon \prod_n\cM_n/\cI\to\prod_n\cN_n/\cJ$ is called \emph{isomorphically coordinate-respecting} if it is coordinate-respecting and $\alpha^\Phi$ is an isomorphism. 
\end{definition}

A function, even a bijective one, between quotients can be coordinate-respecting without being isomorphically coordinate-respecting (see Example~\ref{Ex.Th.CR}). In order to fully appreciate this example, the reader is promptly exposed to additional definitions. 

If a function $\Phi$ is coordinate-respecting, then for all $S\in \cPN/\cI$ we have a function
\[
\Phi_S\colon \left(\prMI\right) \rs S\to \left(\prNJ\right) \rs \alpha(S)
\]
defined by 
\[
\Phi_S(\pi_S(a))=\pi_{\alpha(S)}(\Phi(a)), 
\]
making the following diagram commute:
\begin{center}
 \begin{tikzpicture}
 \matrix[row sep=1cm,column sep=1cm]
 {
 & & \node (M1) {$\prMI$}; && &\node (M2) {$\prNJ$};&
 \\
 & & \node (Q1) {$\left(\prMI\right)\rs S$}; &&& \node (Q2) {$\left(\prNJ\right)\rs \alpha(S)$} ;
 \\
 };
 \draw (M1) edge [->] node [above] {$\Phi$} (M2);
 \draw (Q1) edge [->] node [above] {$\Phi_S$} (Q2);
 \draw (M1) edge [->] node [left] {$\pi_S$} (Q1);
 \draw (M2) edge [->] node [right] {$\pi_{\alpha(S)}$} (Q2);
\end{tikzpicture}
\end{center}
If $\Phi$ is a bijection, so is $\Phi_S$. Further, if all $\cM_n$ and $\cN_n$, for $n\in\bbN$, are $\mathcal L$-structures for some first-order language $\mathcal L$ and if $\Phi$ is a homomorphism, then $\Phi_S$ is a homomorphism.

Functions of twisted product form (Definition~\ref{Def.trivial}) are coordinate-respecting, where the map $\alpha$ is constructed from the function $f\colon\bbN\to\bbN$. Likewise, bijections of twisted product form are isomorphically coordinate-respecting.

\begin{definition}\label{Def.CR.Th}
 A first-order theory $T$ is said to \emph{recognize coordinates} if every isomorphism between reduced products of models of $T$ is isomorphically coordinate-respecting. 
\end{definition}

One could consider an alternative property of $T$, requiring that every \emph{homomorphism} between reduced products of its models is coordinate-respecting. 
It is not difficult to see that the proof of Proposition~\ref{P.Th.CR}, part \eqref{Th.CR.BA}, gives this property. The same can be said for certain rings (extending \eqref{Th.CR.F}), see Remark~\ref{rem:ringshomo}. Apart from these two isolated cases, we regrettably have nothing to say about this notion. Since we will thus focus on the property of being isomorphically coordinate-respecting, we will from here on simply write \emph{coordinate-respecting} for \emph{isomorphically coordinate-respecting}.

The rest of this section is dedicated to providing a (certainly not exhaustive) list of theories satisfying Definition~\ref{Def.CR.Th}.

\begin{proposition}\label{P.Th.CR}
 Each of the following theories recognizes coordinates. 
 \begin{enumerate}
 \item \label{Th.CR.BA} The theory of the two-element Boolean algebra. 
 \item \label{Th.CR.F} The theory of unital rings with no nontrivial central idempotents, and in particular the theory of fields. 
 \item \label{Th.CR.LO} The theory of trees,\footnote{For us, trees are always rooted. Trees do not form an axiomatizable class since well-orderings do not form one. Models of the theory of trees are special examples of \emph{connected ramified sets}, and our result applies to them.} and that of linear orders. 
 \item \label{Th.CR.graph} The theory of random graphs, and in particular the theory of the Rado graph.
 \end{enumerate}
\end{proposition}

\begin{proof}
If $\cM$ is the two-element Boolean algebra then $\prod \cM/\cI\cong \cPN/\cI$, hence \eqref{Th.CR.BA} is vacuous. \eqref{Th.CR.F} is Lemma~\ref{L.Th.CR.F}, \eqref{Th.CR.LO} is Proposition~\ref{P.ramified.CR}, and \eqref{Th.CR.graph} is a special case of Proposition~\ref{prop:GRrigid}.
\end{proof}

\begin{example} \label{Ex.Th.CR} 
The theory of rings, and even the theory of rings of cardinality $\leq 2^n$ for any $n\geq 2$, does not recognize coordinates. More generally, if $T$ is a theory that is preserved under products and has a model with at least two elements (or even if there are models of $T$, all with at least two elements, whose product is a model of~$T$), then $T$ does not recognize coordinates.

To see this, fix models $\cA$ and $\cB$ of $T$ with $|\cA|,|\cB|\geq 2$ such that $\cA\times \cB$ is a model of~$T$. Let $\cM_n\cong \cA\times \cB$, $\cN _{2n}\cong \cA$, and $\cN_{2n+1}\cong \cB$, for $n\in \bbN$. 

Then $\prod_n \cM_n/\Fin$ and $\prod_n \cN_n/\Fin$ are trivially isomorphic via the 
isomorphism $\Phi\colon\prod_n\cM_n/\Fin\to\prod_n\cN_n/\Fin$ that is induced by the fact that $\cM_n$ is isomorphic to $\cN_{2n}\times \cN_{2n+1}$.
Then $\Phi$ is coordinate-respecting, yet is not (isomorphically) coordinate-respecting and the inverse of this (honestly trivial!) isomorphism is not even coordinate-respecting. Similar considerations have been made for reduced products of \cstar-algebras (see \cite[Remark 3.14]{mckenney2018forcing}).

Other examples of theories that do not recognize coordinates include the theory of partial orders, the theory of groups, and the theory of graphs.
Also, if trees are taken with respect to the (fairly commonly used) levelwise product instead of the Cartesian product used here, then the product of two trees is a tree, and therefore with this modification the theory of trees no longer recognizes coordinates. 
\end{example}
Example~\ref{Ex.Th.CR} is yet another example of the fact that isolating the right notion of triviality is nontrivial in certain categories (see \cite[\S 2]{farah2022corona} for a discussion of this phenomenon).

\begin{example}
	We ought to point out that in the domain of continuous (more precisely, affine) logic, a concept analogous to recognizing coordinates has been isolated by Ben Yaacov, Ibarluc\' ia, and Tsankov. They proved that a metric structure whose theory is \emph{simplicial} (meaning that its type spaces are Choquet simplices) can be presented as a direct integral of `extremal' models (\cite[Definition~5.8, Theorem~12.3]{benyaacov2024extremal}). These decompositions, while not unique in the most obvious way (unless the structure is separable), satisfy a uniqueness theorem that suffices to recognize coordinates (see the map ${\mathfrak s}$ in \cite[Theorem~12.15]{benyaacov2024extremal}).
	While averages and direct integrals can be defined for discrete structures (\cite[\S 26]{benyaacov2024extremal}), they behave very differently from (reduced) products considered in the present paper. 
\end{example}

The rest of this section is dedicated to proving Proposition~\ref{P.Th.CR}.
\subsection{Rings with no nontrivial idempotents} 
The following gives the easiest nontrivial example of a theory which recognizes coordinates, even for homomorphisms. It is essentially due to S. Burris (\cite{burris1979sheaf}). We are indebted to Zo\'e Chatzidakis who pointed out the relevance of this example. 

A unital ring $R$ is said to have no nontrivial central idempotents if for all $x\in R$, if $xy=yx$ for all $y\in R$ and $x=x^2$, then $x\in \{0,1\}$. The theory of unital rings which have no nontrivial central idempotents is clearly axiomatizable and a consequence of the theory of fields. 

\begin{lemma}\label{L.Th.CR.F} 
The theory of unital rings with no nontrivial central idempotents recognizes coordinates. 
\end{lemma}

\begin{proof}
Let $T$ be the theory of unital rings with no nontrivial central idempotents. If $\cM_n$, for each $n$, is a model of $T$, then $\cM=\prod_n \cM_n/\cI$ is a ring. The set $\bbI(\cZ(\cM))$ of central idempotents in $\cM$ is definable. Note that $a\in \cM$ is a central idempotent if and only if the set of $n$ such that $a(n)\notin \{0,1\}$ belongs to $\cI$. Moreover, the ring structure on $\cM$ induces a natural Boolean algebra structure on $\bbI(\cZ(\cM))$, with the Boolean algebra operations (i.e.\ $a\wedge b=ab$ and $a\vee b=a+b-ab$, for all $a,b\in\mathbb{I}(\cZ(\cM))$) clearly (quantifier-free) definable in $\cM$. The mapping (slightly abusing the notation and identifying the characteristic function of a subset of $\bbN$ with an element of $\prod_n\cM_n$)
\begin{equation}\label{eq.idempotents}
\cP(\bbN)/\cI\ni X/\cI\mapsto \pi_\cI(\chi_X)\in \bbI(\cZ(\cM))
\end{equation}
is an isomorphism when $\bbI(\cZ(\cM))$ is endowed with this natural Boolean algebra structure. 
If $\cN=\prod_n \cN_n/\cJ$ is another reduced product of models of $T$, and $\Phi\colon \cM\to \cN$ is an isomorphism, then $\Phi[\bbI(\cZ(\cM))]= \bbI(\cZ(\cN))$ and the restriction to $\bbI(\cZ(\cM))$, via the isomorphism in \eqref{eq.idempotents}, gives an isomorphism 
\[
\alpha^\Phi\colon \cP(\bbN)/\cI\to \cP(\bbN)/\cJ. 
\]
Finally, each $a\in \bbI(\cZ(\cM))$ defines a right ideal $a\cM$ of $\cM$. If $a=\pi_\cI(\chi_X)$ then $a\cM=\cM\rs X$ (with the right-hand side as in \eqref{eq.MrsS}), and the quotient map $\pi_X\colon \cM\to \cM\rs X$ is definable, as multiplication by $a$. 
Therefore $\Phi_X$ defined as 
\[
\Phi_X(b)= \Phi(\pi_\cI(\chi_X)b)
\]
is an isomorphism between $\pi_\cI(\chi_X)\cM$ and $\pi_\cJ(\chi_{\alpha^\Phi(X)})\cN$ which makes the diagram in Definition~\ref{def:CR} commute. 
\end{proof}

\begin{remark}\label{rem:ringshomo}
The rigidity showcased above is very strong, leaving the possibility that forcing axioms imply all homomorphisms between reduced products of fields are `trivial' in some sense. (Compare the $\OCA$ lifting theorem of \cite{Fa:AQ}, also Theorem~\ref{T.OCAsharp-Fin}, its noncommutative version proved in \cite{mckenney2018forcing}, and the beautiful theorem on triviality of endomorphisms of the Calkin algebra, \cite{vaccaro2022trivial}.)

In fact, let $T'$ be the theory of unital rings with no nontrivial idempotents. (All fields are models of $T'$.) The proof of Lemma~\ref{L.Th.CR.F} shows that $T'$ recognizes coordinates for homomorphisms. The reason why we cannot extend this argument to homomorphisms between reduced products of models of $T$ is that the image of the center of a ring by a homomorphism is not necessarily contained in the center of the codomain.
\end{remark}

\subsection{Rigidity through local behaviour: linear orders, trees, and random graphs}\label{S.local}

From the model-theoretic point of view, a satisfactory proof that a theory recognizes coordinates would proceed by exhibiting a copy of $\cP(\bbN)/\cI$ as well as the projections $\pi_X$ for $X\in \cP(\bbN)/\cI$ inside $(\prod_n \cM_n/\cI)^{\eq}$ (see \cite[p. 151]{Hodg:Model}). Regrettably, we don't know whether this is the case for trees or random graphs. Instead, we first construct a coherent family of local copies of $\mathcal P(\bbN)/\cI$ and then uniformize it. We will apply this reasoning to ramified sets (see below) and certain particularly interesting classes of graphs.

\subsubsection{Ramified sets: trees, and linear orders}\label{S.ramified}

Following Kurepa (\cite{Kure:Ensembles}), we say that an ordered set $(L,\leq)$ is a \emph{ramified set}, or a pseudotree, if the set of predecessors of every element of $L$ is linear. A ramified set $(L,\leq)$ is said to be \emph{connected} if every two elements of $L$ are compatible, that is, for every $x,y\in L$ there is $z\in L$ such that $z\leq x$ and $z\leq y$.\footnote{Readers familiar with trees as forcing notions may want to take note that their trees, unlike ours, are turned upside down, which causes many elements to be incompatible.} Examples of connected ramified sets include trees, linear orders, and models of the first-order theory of trees. In the language of ordered sets, connected ramified sets are clearly axiomatizable.

\begin{remark}
The reason for considering only connected ramified sets is that it is easy to construct a reduced product of finite disconnected ramified sets which in ZFC has $2^{2^{\aleph_0}}$ automorphisms: If $\cM_n$ is a set with two points and no relation between them, then each $\cM_n$ is a ramified set. If $\cI$ is nonprincipal, $\cM=\prod_{n}\cM_n/\cI$ is a set of size $2^{\aleph_0}$ and no relations, which therefore has $2^{2^{\aleph_0}}$ permutations, and each permutation is an automorphism when $\cM$ is viewed as a ramified set.
\end{remark}
 
The goal of this subsection is to prove the following:
\begin{proposition}\label{P.ramified.CR}
The theory of connected ramified sets recognizes coordinates.
\end{proposition}
In preparations for this proof, we fix connected ramified sets $\cM_n$ and $\cN_n$, together with atomless ideals $\cI$ and $\cJ$. We set $\cM=\prod_n\cM_n/\cI$ and $\cN=\prod_n\cN_n/\cJ$, and let $\Phi\colon \cM\to\cN$ be an isomorphism. We will prove that there is an isomorphism $\alpha=\alpha^\Phi\colon\mathcal P(\bbN)/\cI\to\cP(\bbN)/\cJ$ such that (see Notation~\ref{Notation.2.5} for the definition of $=_S$)
\[
\text{ if }x=_Sy\text{, then }\Phi(x)=_{\alpha(S)}\Phi(y)
\]
 for all $x,y\in\cM$ and $S\in\cP(\bbN)/\cI$.

Each $\cM_n$ is connected, and therefore downwards directed (i.e., every two elements have a common lower bound), and this property is shared by $\cM$. Exact infima might not exist, but if $\inf\{x,y\}$ exists, we denote it by $x\wedge y$. Suprema often do not exist: if $x$ and $y$ are elements of $\cM$, with lifts $\tilde x$ and $\tilde y$ in $\prod_n\cM_n$, $x$ and $y$ have a supremum $z$ if and only if $\tilde x(n)$ and $\tilde y(n)$ are comparable for all but $\cI$ many $n$ (in which case the supremum is given by the sequence $(\max\{\tilde x(n),\tilde y(n)\})_n$).

Fix $x$ and $y$ in $\cM$ and lifts $\tilde x$ and $\tilde y$ in $\prod_n\cM_n$ for $x$ and $y$ respectively. Let 
\[
S_{x,y}:=\{n\mid \tilde x(n)=\tilde y(n)\}/\cI\in\mathcal P(\bbN)/\cI.
\]
While the set $\{n\mid \tilde x(n)=\tilde y(n)\}$ depends on the choice of the lifts $\tilde x$ and $\tilde y$, $S_{x,y}$ does not. It can be viewed as the supremum, in $\cP(\bbN)/\cI$, of $\{S\in\cP(\bbN)/\cI\mid x=_Sy\}$.
Let
\[
B_{x,y}:=\{z\in\cM\mid z\wedge y=x\text{ and } \exists w\, (w\geq z\text{ and }w\geq y)\}.
\]
 This set is definable, and since $\Phi$ is an isomorphism, we have $\Phi[B_{x,y}]=B_{\Phi(x),\Phi(y)}$. On $\cP(\bbN)$ consider the quasi-ordering
 \[
 S\subseteq_\cI T\qquad\text{if and only if}\qquad S\setminus T\in \cI. 
 \]
 
\begin{lemma}\label{L2.14}
Let $x\leq y$ be in $\cM$. Then
\[
B_{x,y}=\{z\in\cM\mid x\leq z\text{ and }\bbN\setminus S_{x,y}\subseteq_\cI S_{x,z}\}.
\]
Consequently, if also $x\leq z$, then $S_{x,y}\subseteq_\cI S_{x,z}$ if and only if $B_{x,y}\subseteq B_{x,z}$.
\end{lemma}
\begin{proof}
Fix $\tilde x$ and $\tilde y$ in $\prod_n\cM_n$ two lifts for $x$ and $y$ respectively.

For one inclusion, fix $z\in B_{x,y}$ with lift $\tilde z\in \prod_n\cM_n$. Since $z\wedge y=x$, then $z\geq x$. If $\neg (\bbN\setminus S_{x,y}\subseteq_\cI S_{x,z})$, then the set $T=\{n\mid \tilde x(n)\notin \{\tilde y(n),\tilde z(n)\}\}$ is $\cI$-positive. Since $z\wedge y=x$, then $\tilde z(n)$ and $\tilde y(n)$ are incomparable for all (but possibly $\cI$ many) $n\in T$, which contradicts the existence of $w\geq y,z$, as each $\cM_n$ is a ramified set.

For the other inclusion, fix $z\in\cM$ with $z\geq x$ and suppose that $\bbN\setminus S_{x,y}\subseteq_\cI S_{x,z}$. It is immediate that $x=y\wedge z$. Let $\tilde S\in\cP(\bbN)$ be any lift of $S_{x,y}$, and define $\tilde w\in\prod_n\cM_n$ by 
\[
\tilde w(n)=\begin{cases}\tilde z(n)& \text{ if }n\in \tilde S \\
\tilde y(n) & \text{ otherwise.}
\end{cases}
\]
This gives that $\pi_\cI(\tilde w)$ is the supremum of $y$ and $z$.
\end{proof}

By Lemma~\ref{L2.14}, given $x\in\cM$, the definable set $B_{x,y}$ remains invariant under changing the value of $y\geq x$ as long as the value of $S_{x,y}\in\cP(\bbN)/\cI$ remains unchanged. 

Let
\[
D=\{x\in\cM\mid \exists y\in\cM\, (x\leq y\wedge B_{x,y}=\{x\})\}.
\]
This is a definable set, as $B_{x,y}$ is, and one can note that 
\[
D=\{x\in\cM\mid \{n\mid \tilde x(n)\text{ is a terminal node}\}\in\cI\}.
\] 
This gives that if $x\in D$ and $S\in\cP(\bbN)/\cI$ is given, there is $z\in\cM$ such that $S_{x,z}=S$. 
As each $\cM_n$ has more than one point, and $\cM_n$ is connected, $D$ is nonempty. Furthermore, for every $x\in\cM$ there is $y\in D$ with $y\leq x$. (Choose $\tilde y(n)<\tilde x(n)$ or $\tilde y(n)=\tilde x(n)$, if $\tilde x(n)$ is the minimal element of $\cM_n$, where $\tilde x$ is a lift of $x$.)

For $x\in D$, define
\[
\alpha_x\colon\cP(\bbN)/\cI\to\cP(\bbN)/\cJ
\]
in the following way: for every $S\in\cP(\bbN)/\cI$, we find $z\geq x$ such that $S=S_{x,z}$, and we map $S$ to $S_{\Phi(x),\Phi(z)}$. This is well defined (i.e., it does not depend on the choice of $z\geq x$ with $S=S_{x,z}$). Also, since $\Phi(x)\in D$, for every $T\in\cP(\bbN)/\cJ$ there is $z\geq x$ such that $T=S_{\Phi(x),\Phi(z)}$, so the map $\alpha_x$ is surjective. Furthermore, using that $\Phi$ is an isomorphism and Lemma~\ref{L2.14} we have for all $x,y,z\in\cM$ with $x\leq y,z$:
\[
S_{x,y}\subseteq S_{x,z}\Leftrightarrow B_{x,y}\subseteq B_{x,z}\Leftrightarrow
B_{\Phi(x),\Phi(y)}\subseteq B_{\Phi(x),\Phi(z)}\Leftrightarrow S_{\Phi(x),\Phi(y)}\subseteq S_{\Phi(x),\Phi(z)}. 
\]
Therefore $\alpha_x$ is an isomorphism of Boolean algebras, and, if $x\in D$ and $y\in\cM$ is such that $x\leq y$, then
\[
 x=_S y\text{ if and only if }\Phi(x)=_{\alpha_x(S)}\Phi(y). 
\]
We are left to show that the local maps cohere.

\begin{lemma}
For every $x$ and $y$ in $D$, $\alpha_x=\alpha_y$. 
\end{lemma}
\begin{proof}
Since each $\cM_n$ is connected, for every $x$ and $y$ there is $z$ such that $z\leq x$ and $z\leq y$. If $x$ and $y$ are in $D$, so is $z$. Therefore it is enough to show that $\alpha_x=\alpha_y$ if $x\leq y$.

\begin{claim}
For each $x\leq y$ in $D$, we have that $\alpha_x\restriction S_{x,y}=\alpha_y\restriction S_{x,y}$.
\end{claim}
\begin{proof}
Let $S=S_{x,y}$. If $T\subseteq \bbN$ satisfies $T\subseteq_\cI S$ and $z\geq y$ is such that $S_{y,z}=T$, then $S_{x,z}=T$. Since $z=_Tx$, then $\Phi(z)=_{\alpha_x(T)}\Phi(x)$. Likewise, $\Phi(z)=_{\alpha_y(T)}\Phi(y)$. Since $\Phi(y)=_{\alpha_x(S)}\Phi(x)$, we have that $\alpha_x(T)\subseteq_{\cJ}\alpha_y(T)$. Since $\Phi$ is an isomorphism, and $\alpha_{\Phi(x)}=\alpha_x^{-1}$, by working with $\Phi^{-1}$ instead, we get that for all $T\subseteq_{\cJ} \alpha_x(S)$, $\alpha^{-1}_x(T)\subseteq_{\cI} \alpha^{-1}_y(T)$. Combining the two, we get that $\alpha_x(T)=\alpha_y(T)$ whenever $T\subseteq_{\cI} S$.
\end{proof}

Now, pick $x$ and $y$ in $D$ with $x\leq y$. Suppose $\alpha_x\neq\alpha_y$. Since $\alpha_x$ and $\alpha_y$ are isomorphisms, we can find a nonzero $S\in\cP(\bbN)/\cI$ such that $\alpha_x(S)\cap\alpha_y(S)=\emptyset$ (here the intersection is computed in the Boolean algebra $\mathcal P(\bbN)/\cJ$). Notice that, by the above claim $S\cap S_{x,y}=\emptyset$ (in $\cP(\bbN)/\cI)$. 

Let $z$ be such that $z=_Sx$ and $z=_{\bbN\setminus S}y$, so that $x\leq z\leq y$, and $z\in D$. By the above claim, $\alpha_z(S)=\alpha_x(S)\subseteq_\cJ \alpha_y(\bbN\setminus S)=\alpha_z(\bbN\setminus S)$; this is a contradiction.
\end{proof}

\begin{proof}[Proof of Proposition~\ref{P.ramified.CR}]
Let $\alpha=\alpha_x$, for some (any) $x\in D$. We claim that $\alpha$ is the required isomorphism.

For $x\in \cM$, let 
\[
T_x=\{n\mid \tilde x(n)\text{ is a terminal node}\}/\cI.
\]
Clearly, $T_x$ does not depend on the choice of the lift $\tilde x$.

\begin{claim}\label{claim1tree}
For all $x$ and $y$ in $\cM$, if $S\cap T_x=\emptyset$ (in $\cP(\bbN)/\cI$) and $x=_Sy$, then $\Phi(x)=_{\alpha(S)}\Phi(y)$.
\end{claim}

\begin{proof}
We can find $w\in D$ with $w\leq x$ and $w\leq y$ such that $w=_Sx$. Since $\alpha=\alpha_w$, we have that $\Phi(x)=_{\alpha(S)}\Phi(w)=_{\alpha(S)}\Phi(y)$.
\end{proof}
\begin{claim}\label{claim2tree}
For every $x\in\cM$, $\alpha(T_x)=T_{\Phi(x)}$.
\end{claim}
\begin{proof}
Fix $x$. Let $w\leq x$ such that $w\in D$ and $w=_{\bbN\setminus T_x}x$. Since $(\bbN\setminus T_x)\cap T_x=\emptyset$, by the above claim, $\Phi(w)=_{\bbN\setminus \alpha(T_x)}\Phi(x)$. Since $T_{\Phi(w)}\in\cI$, as $D$ is definable, we have that $\bbN\setminus \alpha(T_x)$ is disjoint from $T_{\Phi(x)}$, hence $\alpha(T_x) \subseteq T_{\Phi(x)}$. Since $\Phi$ and $\alpha$ are isomorphisms and $x$ is arbitrary, applying the same argument to $\Phi^{-1}$, $\alpha^{-1}$ and $\Phi^{-1}(x)$ gives the thesis.
\end{proof}

We are ready to conclude our proof. Fix $x$ and $y$ and suppose that $x=_Sy$. By Claim~\ref{claim1tree}, we have that $\Phi(x)=_{\alpha(S\setminus T_x)}\Phi(y)$, so we can assume that $S\subseteq_\cI T_x$. Fix~$w$ such that $w\leq x$, $w\leq y$, and $T_w=S$. (Such $w$ can be constructed by taking $x$ on $\tilde S$ and nonterminal nodes lower than both $\tilde x(n)$ and $\tilde y(n)$ on $\bbN\setminus \tilde S$, where $\tilde S$ is a lift for $S$ in $\cP(\bbN)$.) It is enough to show that $\Phi(x)=_{\alpha(S)}\Phi(w)$. Fix $a$ and $b$ in $\prod_n\cN_n$ such that $\pi(a)=\Phi(x)$ and $\pi(b)=\Phi(w)$. If $\Phi(x)\neq_{\alpha(S)}\Phi(w)$, there is an $\cI$-positive $S'\subseteq \tilde S$ such that $a(n)\neq b(n)$ whenever $n\in \alpha(S')$. 

Since $S\subseteq_\cI T_x\cap T_w$, by Claim~\ref{claim2tree},
\[
\alpha(S')\subseteq_{\cJ}\alpha(T_x\cap T_w)=\alpha(T_x)\cap \alpha(T_w)=T_{\Phi(x)}\cap T_{\Phi(w)},
\]
and so every node in $\alpha(S')$ is terminal for both $a$ and $b$. Since these are different, $a$ and $b$ are not comparable on the positive set $\alpha(S')$, and so $\Phi(x)=\pi(a)$ is not comparable with $\Phi(w)=\pi(b)$. This is a contradiction, since $w\leq x$ implies that $\Phi(w)\leq \Phi(x)$. Hence $\Phi(w)=_{\alpha(S)}\Phi(x)$, and, by replacing $x$ with $y$ above, $\Phi(w)=_{\alpha(S)}\Phi(y)$, which implies that $\Phi(x)=_{\alpha(S)}\Phi(y)$. This concludes the proof.
\end{proof} 

\subsubsection{Graphs}\label{S.graphs}
In this subsection we prove Proposition~\ref{P.Th.CR}(\ref{Th.CR.graph}). The language of the theory of graphs has a single binary relation $E$ which codes adjacency. If $\cM_n$ are graphs then so is $\prod_n \cM_n/\cI$, and two vertices $a$ and $b$ are adjacent if and only if
\begin{equation}\label{eq.adjacent}
(\forall ^{\cI} n)\,E^{\cM_n}(\tilde a(n),\tilde b(n)).
\end{equation}
Recall that a graph is simple if it has no loops and no multiple edges. 
If the graphs $\cM_n$ are simple (and we will consider only simple graphs) then \eqref{eq.adjacent} is equivalent to 
\begin{equation}\label{eq.simple}
(\forall^\cI n) \, E^{\cM_n}(\tilde a(n),\tilde b(n))\wedge \tilde a(n)\neq \tilde b(n).
\end{equation}
Suppose now that $\cM_n$, for $n\in\bbN$, are just sets, and that $\cI$ is an atomless ideal on $\bbN$. Fix $a,b\in \prod_n \cM_n/\cI$ and $S\in\cP(\bbN)/\cI$ (with lifts $\tilde a,\tilde b$ and $\tilde S$). Define $\tilde\chi^{\tilde a\tilde b}_{\tilde S}$ by
\[ 
\tilde \chi_{\tilde S}^{\tilde a\tilde b}(n)=\left\{
\begin{array}{ll}
\tilde a(n)\quad &\text{ if }n\notin \tilde S\\
\tilde b(n)\quad &\text{ if }n\in \tilde S.
\end{array}
\right.
\]
Let 
\[
\chi_S^{ab}=\pi_\cI( \tilde \chi_{\tilde S}^{\tilde a\tilde b}).
\]
The definition of $\chi_S^{ab}$ does not depend on the choice of the lifts. Let 
\begin{equation}\label{eq.Aab}
A_{ab}=\{ \chi_S^{ab} \mid S \in \mathcal P(\bbN)/\cI \}.
\end{equation}
Write 
\begin{equation}\label{eq.perp}
a\perp b\text{ if }\pi_S(a)\neq \pi_S(b) \text{ for all nonzero }S\in \cPN/\cI,
\end{equation}
or equivalently, if $\forall^\cI n\,(\tilde a(n)\neq \tilde b(n))$.
 
If $a\perp b$ then the function $\cPN/\cI\ni S\mapsto \chi^{ab}_S\in A_{ab}$ is a bijection, and $A_{ab}$ is equipped with a canonical Boolean algebra structure, obtained when $\chi^{ab}_S$ is identified with $S\in \cP(\bbN)/\cI$. 

We are interested in the situation where the relation $\perp$ and the Boolean algebra structure on $A_{ab}$ are definable. Note that this is not always the case: 

\begin{example}\label{Ex.NullGraph}
If each $\cM_n$ is a graph with at least two vertices and $\cM_{n}$ is the null graph for infinitely many $n$, then for any two distinct vertices $a$ and $b$ in $\cM=\prod_n \cM_n/\Fin$ \eqref{eq.simple} fails and therefore $\cM$ is also a null graph. In this case every definable set in $\cM$ is finite or cofinite. Therefore neither the sets $A_{ab}$ nor the relation $\perp$ are definable and $\cM$ has $2^{2^{\aleph_0}}$ automorphisms in $\ZFC$. 
\end{example}

The \emph{random graph} $\mathcal R$ (also known as the \emph{Rado graph}) is the unique graph with countable set of vertices such that for every two disjoint finite sets of vertices $P$ and~$Q$ there is a vertex $v$ adjacent to all vertices in $P$ and not adjacent to any vertex in $Q$. 

\begin{definition}\label{Def.SuffRG}
Let $k$ be a positive natural number. We say that a graph $G$ is \emph{$k$-sufficiently random} if for every set $P$ of $k$ vertices and every vertex $u\notin P$ there is a vertex $v\neq u$ that is not adjacent to $u$ and is adjacent to all vertices in $P$ (the latter property implies $v\notin P$). 

We say that $G$ is \emph{sufficiently random} if it is 2-sufficiently random and it has at least three vertices.
\end{definition}

The class of sufficiently random graphs is clearly axiomatizable in the usual language of graphs, where there is a single binary relation $E$ coding adjacency.

The canonical models for randomness introduced by Gilbert (\cite{GilbertGraphs}, see also a general 
reference on random graphs such as e.g.\ \cite{frieze2016introduction}) are sufficiently random. To wit, if $n\geq 1$ and $0<p<1$, then $\bbG_{n,p}$ is a probability space of objects with $n$ vertices, and associated to the $n \choose 2$ pairs of these vertices we have independent random variables $x_{ij}$, $1\leq i<j\leq n$, such that $x_{ij}=1$ with probability $p$. Given $x_{ij}$, vertices $i$ and $j$ are connected if and only if $x_{ij}=1$. This describes a probability space on the space of all simple graphs with~$n$ vertices, and as $n\to \infty$ the probability that a graph in this space is sufficiently random converges to 1. 
The following is a consequence of standard facts about random graphs and it will play a role in the proof of Theorem~\ref{T.Graphs}. 

\begin{lemma} \label{L.sufficiently-random}
For every $0<p<1$, the probability that $\bbG_{n,p}$ is not sufficiently random converges to 0 as $n\to \infty$. 
\end{lemma} 

\begin{proof} 
For $s\geq 1$, a graph is said to have the full level $s$ extension property if for any two disjoint sets of vertices $F_1$ and $F_2$ such that $|F_1|+|F_2|\leq s$ there is a vertex $v$ adjacent to all vertices in $F_1$ and not adjacent to any of the vertices in $F_2$. Note that if $s\geq 3$ and we add the requirement that $v$ does not belong to $F_2$ then every graph with at least three vertices and this strengthened version of the full level $s$ extension property is sufficiently random. 
	
The proof of \cite[Theorem~10.7.5]{alon2008probabilistic} (this is Theorem~10.4.5 in the 3rd edition of the book) shows that for any $s$ and $0<p<1$, the probability that $\bbG_{n,p}$ does not have the version of the full level $s$ extension property in which it is required that $v\notin F_2$ is (writing $\varepsilon=\min(p,1-p)^s$) not greater than $s^2n^s(1-\varepsilon)^{n-s}$. An inspection of the proof (paying attention to the exponent $n-a-b$ in the third displayed formula) shows that this is also an estimate that $\bbG_{n,p}$ does not have the strengthened full level $s$ extension property. 
\end{proof} 

The remaining part of the section is dedicated to a proof of the following strengthening of Proposition~\ref{P.Th.CR}(\ref{Th.CR.graph}).

\begin{proposition}\label{prop:GRrigid}
The theory of sufficiently random graphs recognizes coordinates. 
\end{proposition}

We start by showing that the sets $A_{ab}$ and the relation $\perp$ are definable in this setting.

\begin{lemma} \label{L.Aab.perp}
Suppose that $\cM=\prod_n \cM_n/\cI$ is a reduced product of sufficiently random graphs. Then for all $a,b\in\cM$ we have that 
\[
x\in A_{ab}\text{ if and only if } \forall z\, (E(z,a)\wedge E(z,b)\Rightarrow E(z,x)).
\]
If in addition $a$ and $b$ are distinct, then
\[
a\perp b\text{ if and only if }\forall x\, (b\in A_{xa}\Rightarrow b=x).
\]
\end{lemma}
\begin{proof}
To prove the direct implication in the first assertion, assume $x\in A_{ab}$. Then $(\forall^\cI n) \,\tilde x(n)\in \{\tilde a(n),\tilde b(n)\}$. If $z\in\cM$ satisfies $E(z,a)$ and $ E(z,b)$, then the characterization of adjacency given in \eqref{eq.adjacent} implies $E(z,x)$. 

For the converse implication, assume $x\notin A_{ab}$. Then $S=\{n\mid \tilde x(n)\notin \{\tilde a(n),\tilde b(n)\}\}$ is $\cI$-positive. For each $n\in S$, by applying the definition of sufficiently random to $P=\{\tilde a(n),\tilde b(n)\}$ and $u=\tilde x(n)$, there is $\tilde z(n)\in \cM_n\setminus\{\tilde x(n)\}$ such that $\lnot E(\tilde z(n),\tilde x(n))$, $E(\tilde z(n),\tilde a(n))$, and $E(\tilde z(n),\tilde b(n))$. Since $S$ is $\cI$-positive, this implies $\lnot E(z,x)$, hence the right-hand side of the first assertion fails. 

For the second assertion, fix distinct $a$ and $b$ in $\cM$ and assume that the right-hand side fails. Thus there is $x\neq b$ such that $b\in A_{xa}$. Thus $b=\chi_S^{xa}$ for some $S\in \cP(\bbN)/\cI$. Since $x\neq b$, $S\neq 0$. Therefore $\pi_S(a)=\pi_S(b)$ and $a\not\perp b$. 

Vice versa, suppose that $a\not\perp b$, and let $\tilde S=\{n\mid\tilde a(n)=\tilde b(n)\}$. Since $a$ and $b$ are distinct and $a\not\perp b$, both $\tilde S$ and $\bbN\setminus\tilde S$ are $\cI$-positive. Let $\tilde c$ be such that $\tilde c(n)\neq\tilde a(n)$ for all $n\in \tilde S$ and $\tilde c(n)=\tilde b(n)$ for $n\notin\tilde S$. Then $\tilde b= \tilde \chi_{\tilde S}^{\tilde c\tilde a}$, and so $b\in A_{ca}$, but clearly $b\neq c$.
\end{proof}

\begin{lemma} \label{L.cZ} 
If $\cM$ is a reduced product of sufficiently random graphs then let 
\[
\cZ=\cZ^{\cM}=\{(a,b)\mid a\perp b\}.
\]
Then for $(a,b)\in\cZ$, the Boolean algebra structure on $A_{ab}$ which is induced by the bijective correspondence 
$\cP(\bbN)/\cI\ni S\mapsto \chi_S^{ab}\in A_{ab}$ is definable in $\prod_n \cM_n/\cI$.
\end{lemma}
\begin{proof} 
Fix $(a,b)\in \cZ$. The set $A_{ab}$ itself is definable from parameters $a$ and $b$ by Lemma~\ref{L.Aab.perp}. 

If $S,S'\in\cP(\bbN)/\cI$, then $S\subseteq S'$ if and only if $\chi_{S'}^{ab}\in A_{b,\chi_S^{ab}}$. The complement of $\chi_{S}^{ab}$ in $A_{ab}$ is the only $x\in A_{ab}$ such that $\chi_S^{ab}\perp x$, that is $\chi_{\bbN\setminus S}^{ab}$. 	
Thus for $(a,b)\in \cZ$ the set $A_{ab}$ is definable and so are the order and the complement in the Boolean structure that it carries, all using $a$ and $b$ as parameters. Since the union of two elements of a Boolean algebra is their minimal upper bound and all Boolean operations can be expressed in terms of union and complement, this completes the proof. 
\end{proof}

\begin{lemma} \label{L.alpha}
Consider reduced products $\cM=\prod_n\cM_n/\cI$ and $\cN=\prod_n\cN_n/\cJ$ of sufficiently random graphs associated with atomless ideals $\cI$ and $\cJ$ and an isomorphism $\Phi\colon\cM\to\cN$. Then there is an isomorphism $\alpha \colon\cP(\bbN)/\cI\to\cP(\bbN)/\cJ$ such that 
\[
\Phi(\chi_{S}^{ab})=\chi_{\alpha(S)}^{\Phi(a)\Phi(b)} 
\]
for all $(a,b)\in \cZ^{\cM}$ and all $S\in \cP(\bbN)/\cI$. 
\end{lemma}

\begin{proof}
The proof proceeds in several steps. First, fix $(a,b)\in \cZ^\cM$. Since the 	Boolean algebra structure on $A_{ab}$ is definable by Lemma~\ref{L.cZ}, there is an isomorphism 
\[
\alpha^{ab}\colon\cP(\bbN)/\cI\to\cP(\bbN)/\cJ
\]
obtained by sending $S$ to the unique $T$ such that $\Phi(\chi_{S}^{ab})=\chi_{T}^{\Phi(a)\Phi(b)}$. 

Second, all $(a,b)\in \mathcal Z$ satisfy the symmetry property $\alpha^{ab}=\alpha^{ba}$. This is because all $S\in\mathcal P(\bbN)/\cI$ and $T\in\cP(\bbN)/\cJ$ satisfy $\chi_{S}^{ab}=\chi_{\bbN\setminus S}^{ba}$ and $\chi_{T}^{\Phi(a)\Phi(b)}=\chi_{\bbN\setminus T}^{\Phi(b)\Phi(a)}$. 
 
Next, fix $a,b$ and $c$ in $\cM$ with $(a,b)$ and $(a,c)$ in $\cZ$. We claim that $\alpha^{ab}=\alpha^{ac}$. Assume otherwise and fix an $\cI$-positive $S$ such that $\alpha^{ab}(S)\cap\alpha^{ac}(S)=\emptyset$. Lift $S$ to $\tilde S\in \cP(\bbN)$. Let $d$ be constructed in such a way that: 
\begin{itemize}
\item for all $n\in\tilde S$, $\tilde d(n)$ is connected to both $\tilde b(n)$ and $\tilde c(n)$ but not to $\tilde a(n)$,
\item for all $n\notin\tilde S$, $\tilde d(n)$ is connected to $\tilde a(n)$.
\end{itemize}
Such a $d$ exists as each graph is sufficiently random. Note that $\neg E(a,d)$. On the other hand, we have $E(d,\chi_S^{ab})$ and $E(d,\chi_S^{ac})$, hence $E(\Phi(d),\chi_{\alpha^{ab}(S)}^{\Phi(a)\Phi(b)})$ and $E(\Phi(d),\chi_{\alpha^{ac}(S)}^{\Phi(a)\Phi(c)})$. As $\alpha^{ab}(S)$ and $\alpha^{ac}(S)$ are disjoint, the union of their complements is (modulo $\cJ$) equal to $\mathbb{N}$. Lifting $\Phi(d), \chi_{\alpha^{ac}(S)}^{\Phi(a)\Phi(c)}, \chi_{\alpha^{ab}(S)}^{\Phi(a)\Phi(b)}$ respectively, and using the fact that $E(x,y)$ if and only if $\forall^\cJ n\, (E(\tilde x(n),\tilde y(n)))$, we have that $E(\Phi(a),\Phi(d))$, which is a contradiction. 
 
Therefore the maps $\alpha^{ab}$ satisfy a transitivity property: if $(a,b)$ and $(a,c)$ in $\cZ$ then $\alpha^{ab}=\alpha^{ac}$.
		
We are ready to prove that for arbitrary $(a,b)$ and $(c,d)$ in $\cZ$, $\alpha^{ab}=\alpha^{cd}$.
Fix arbitrary $(a,b)$ and $(c,d)$ in $\cZ$. Since each $\cM_n$ has at least three points we can find $e\in\cM$ such that $(a,e)$ and $(e,c)$ are in $\cZ$. By the transitivity and symmetry properties we have $\alpha^{ab}=\alpha^{ae}=\alpha^{ce}=\alpha^{cd}$. Since $(a,b)$ and $(c,d)$ were arbitrary, setting $\alpha=\alpha^{ab}$ for some (any) $(a,b)\in\cZ$ we have the required isomorphism. 
\end{proof}

We are ready to conclude our proof.

\begin{proof}[Proof of Proposition~\ref{prop:GRrigid}] 
Fix sufficiently random graphs $\cM_n$ and $\cN_n$ for $n\in\bbN$, atomless ideals $\cI$ and $\cJ$, and an isomorphism
\[
\Phi\colon\prod_n\cM_n/\cI \to\prod_n\cN_n/\cJ.
\]
We first show that for all $x,y\in\cM$ there exists $(a,b)\in\cZ$ such that $\{ x,y \}\subseteq A_{ab}$. 
Fix $x$ and $y$ in $\cM$. If $x\perp y$, then $(x,y)\in\cZ$ and obviously $x,y\in A_{xy}$. If $x\not\perp y$, let $\tilde S=\{n\mid \tilde x(n)=\tilde y(n)\}$. For $n\in \tilde S$, let $\tilde z(n)$ be any element of $\cM_n$ different from $\tilde x(n)$, and if $n\notin\tilde S$, let $\tilde z(n)=\tilde x(n)$. Then $z\perp y$, and $x,y\in A_{yz}$.

Now, suppose that $x$ and $y$ are such that (using Notation~\ref{Notation.2.5}) $x=_Sy$ for some $S\in\cP(\bbN)/\cI$. As $\perp$ is definable, we can assume that $S$ is nonzero. Let $a$ and $b$ such that $x,y\in A_{ab}$. Let $S_1$ and $S_2$ be such that $x=\chi_{S_1}^{ab}$ and $y=\chi_{S_2}^{ab}$. Since $x=_S y$, then $(S\cap S_1) \Delta (S\cap S_2)\in\cI$. Using $\alpha\colon \cP(\bbN)/\cI\to \cP(\bbN)/\cJ$ provided by Lemma~\ref{L.alpha} we have 
\[
\Phi(x)=_{\alpha(S)\setminus\alpha(S_1)}\Phi(a)\text{ and }\Phi(x)=_{\alpha(S)\cap\alpha(S_1)}\Phi(b)
\]
and
\[
\Phi(y)=_{\alpha(S)\setminus\alpha(S_2)}\Phi(a)\text{ and }\Phi(y)=_{\alpha(S)\cap\alpha(S_2)}\Phi(b).
\]
Since $\alpha$ is a homomorphism, $S\cap S_1=S\cap S_2$ implies $\alpha(S)\cap\alpha(S_1)=\alpha(S)\cap\alpha(S_2)$ and $\alpha(S)\setminus\alpha(S_1)=\alpha(S)\setminus \alpha(S_2)$. This implies that $\Phi(x)=_{\alpha(S)}\Phi(y)$, that is, the thesis.
\end{proof}

\subsection{A general local criterion}\label{S.local2}
Propositions~\ref{P.ramified.CR} and \ref{prop:GRrigid} have a similar pattern, which can be generalized as follows.

Fix sets $\cM_n$, for $n\in\bbN$, and let $\cI\subseteq\mathcal P(\bbN)$ be an ideal. Let $\cM=\prod_n\cM_n/\cI$. If $a$ and $b$ are in $\cM$ and $S\in\cP(\bbN)$, as in \S\ref{S.graphs}, let $\chi_S^{ab}$ be the unique element of $\cM$ such that 
\[
\pi_S(\chi_S^{ab})=\pi_S(b)\text{ and }\pi_{\bbN\setminus S}(\chi_S^{ab})=\pi_{\bbN\setminus S}(a).
\]
Let $A_{ab}=\{\chi_S^{ab}:S\in\cP(\bbN)/\cI\}$, and write $a\perp b$ if $\pi_S(a)\neq\pi_S(b)$ for every nonzero $s\in\cP(\bbN)/\cI$.

If $a\perp b$, we have a canonical Boolean algebra structure on $A_{ab}$ given by the isomorphism between $A_{ab}$ and $\cP(\bbN)/\cI$, obtained by mapping $S\to\chi_S^{ab}$. (Notice that if $a\not\perp b$, then this map is not injective.) In both Propositions~\ref{P.ramified.CR} and \ref{prop:GRrigid}, the sets $A_{ab}$, for $a\perp b$, as well as their canonical Boolean algebra structures, were definable. 

Fix a theory $T$. By the Feferman--Vaught theorem (\cite{feferman1959first} or \cite[Proposition~6.3.2]{ChaKe}), there is a theory $T^*$ such that all reduced products of models of $T$ over atomless ideals are models of $T^*$.

In the following definition, it is required that the set $\cZ$ is first-order definable and that there is a first-order definition of a system of Boolean algebras that is correctly interpreted in every reduced product of models of $T$. While the theory $T^*$ has other models, we do not put any requirements on the interpretation of these first-order definitions in models of $T^*$ that are not reduced products. 

\begin{definition} \label{def:locallyCR}
Let $T$ be a first order theory, and suppose that $\cZ$ is a set of pairs definable in $T^*$. We say that $T$ \emph{recognizes coordinates at} $\cZ$ if for all $(a,b)\in\cZ$ we have that $a \perp b$ and the sets $A_{ab}$ (as in \eqref{eq.Aab}), as well as their canonical Boolean algebra structure, are uniformly definable\footnote{`Uniformly definable' means that a single formula with parameters $a$ and $b$ defines $A_{ab}$ for all relevant pairs $a,b$ in all relevant reduced products, and that analogous fact applies to Boolean operations in the canonical Boolean algebra structure on $A_{ab}$.} in all reduced products of models of $T$ over atomless ideals.
\end{definition}

The proof of the following lemma is left as an (easy) exercise for the interested reader.
\begin{lemma} \label{lemma:localcopies}
Let $T$ be a theory that recognizes coordinates at a nonempty definable set $\cZ$. Let $\cM_n$ and $\cN_n$ be models of $T$, and let $\cI$ and $\cJ$ be atomless ideals on $\bbN$. 
Let $\cM=\prod_n\cM_n/\cI$ and $\cN=\prod_n\cN_n/\cJ$, and suppose that $\Phi\colon \cM\to\cN$ is an isomorphism. 

Then for all $(a,b)\in \cZ^\cM$ there is an isomorphism 
\[
\alpha^{ab}\colon\cPN/\cI\to \cPN/\cJ
\]
 such that every $S\in \cPN/\cI$ satisfies
 \[
 \Phi(\chi^{ab}_S)=\chi^{\Phi(a)\Phi(b)}_{\alpha^{ab}(S)}.\eqno\qed
 \]
\end{lemma} 

As in the proof of Propositions~\ref{P.ramified.CR} and \ref{prop:GRrigid}, we need the local isomorphisms to cohere. Once again, we leave this proof as an exercise.
\begin{proposition}\label{prop:glueingcopies} 
In the setting of Lemma~\ref{lemma:localcopies}, suppose that
\begin{itemize}
\item $\alpha^{ab}=\alpha^{cd}$ for every $(a,b),(c,d)\in\cZ^{\cM}$, and
\item for all $x,y\in\cM$ there are $a,b$ with $(a,b)\in\cZ^{\cM}$ and $x,y\in A_{ab}$.
\end{itemize} 
Then $T$ recognizes coordinates. \qed
\end{proposition}

\section{Open Colouring Axioms}\label{S.OCAsharp}
In this section, we introduce a self-strengthening of the Open Colouring Axiom. More precisely, Definition~\ref{Def.OCAsharp} below gives a sharpening of the reformulation of the axiom $\OCAT$ introduced in \cite{Fa:Cauchy} that will (in spite of being a consequence of $\OCAT$) considerably simplify some of the uniformization arguments.

\begin{notation}\label{Notation}
For distinct $a$ and $b$ in $ \twoN$ let 
\begin{align*}
\Delta(a,b)&=\min\{n\mid a(n)\neq b(n)\}\text{ and }\\
a\wedge b&=a\rs \Delta(a,b). 
\end{align*}
In particular, $\Delta(a,b)=|a\wedge b|$. 
If $Z\subseteq \twoN$ let
\[
\Delta(Z)=\{a\wedge b\mid a,b\text{ are distinct elements of $Z$}\}. 
\]
For $s\in\{0,1\}^{<\bbN}$ let
\[
[s]=\{a\in \twoN\mid a(i)=s(i)\text{ for all }i<|s|\}.
\]
 Using this notation we have an alternative (equivalent) definition of $\Delta(Z)$: 
\[
\Delta(Z)=\{s\in \twolN\mid [s^\frown i]\cap Z\neq \emptyset\text{ for }i=0,1\}.
\] 
\end{notation}
\
In \cite{Fa.Luzin}, the second author defined an apparent strengthening of $\OCAT$ named $\OCAinfty$.\footnote{In \cite{Fa:Cauchy} a weaker axiom was called $\OCAinfty$, although a proof that $\OCAinfty$ as introduced in \cite{Fa.Luzin} is a consequence of $\PFA$ was given. (We should note that this proof contains a glaringly obvious typo; in 4. of Lemma~3.1, `whenever $s\subseteq t$’ should be `for some $t$ such that $s\subseteq t$.’)} This axiom was proved to be equivalent to $\OCAT$ (\cite[\S 5]{moore2021some}, see also \cite[Theorem 8.6.6]{Fa:STCstar}). $\OCAsharp$ can be viewed as a refined version of $\OCAinfty$, and Moore’s proof that $\OCAT$ implies $\OCAinfty$ easily adapts to prove that $\OCAT$ implies $\OCAsharp$. 

The space of unordered pairs $[X]^2$ in a metric space $X$ is topologized as follows. Consider the diagonal,
\[
\Diag_X=\{(x,x) \mid x \in X \}
\] 
and identify symmetric subsets of $X^2 \setminus \Diag_X$ with subsets of $[X]^2$. We call a subset of $[X]^2$ open if the associated symmetric subset of $X^2 \setminus \Diag_X$ is open.

\begin{definition} \label{Def.OCAsharp} 
$\OCAsharp$ is the following statement. Suppose that $X$ is a separable metric space and that $\cV_j$ is a countable family of symmetric open subsets of $X^2\setminus \Diag_X$ such that $\bigcup\cV_j\supseteq \bigcup\cV_{j+1}$ for all $j\in \bbN$. 
Then one of the following alternatives holds. 
\begin{enumerate}
\item \label{OCAsharp.1} There are sets $X_n$, for $n\in \bbN$, such that $X=\bigcup_n X_n$ and 
\[
[X_n]^2\cap\bigcup\cV_n =\emptyset
\] 
for all $n$. 
\item\label{OCAsharp.2} There are an uncountable $Z\subseteq \twoN$, an injective $f\colon Z\to X$, and $\rho\colon \Delta(Z)\to \bigcup_j \cV_j$\footnote{The reader will hopefully forgive us for pointing out the obvious, that $\bigcup\cV_j$ and $\bigcup_j \cV_j$ are two very different sets.} such that $\rho(s)\in \cV_{|s|}$ for all $s$ and all distinct $a$ and $b$ in $Z$ satisfy
\[
\{f(a),f(b)\}\in \rho(a\wedge b). 
\]
\end{enumerate}
\end{definition}

To make the relation between $\OCAsharp$ and $\OCAT$ (in the form of $\OCAinfty$) more obvious, the reader may want to take a look at the reformulation of $\OCAT$ in terms of uniformization of coherent families of functions given in \cite[Proposition~2.2.11]{Fa:AQ}. 
With $K_0^n=\bigcup\cV_n$, the first alternatives of $\OCAinfty$ and $\OCAsharp$ are equivalent, and the second alternative of the latter is a sharper variant of the second alternative of the former in which one keeps track of the witness (typically a basic open set), provided by $\rho$, for the reason why $\{f(a), f(b)\}$ belongs to $K^n_0$. 

Our isolation of the not-so-easy-to-parse $\OCAsharp$ is, in our humble opinion, compensated by considerable reduction in complexity in the proof of the $\OCA$ lifting theorem of \cite{Fa:AQ} (see \S\ref{S.Uniformization.Fin}), as well as the proof of our main result (as compared with our original proof that will mercifully not appear anywhere), see~\S\ref{S.Uniformization}.

Our proof of the following is based on \cite[\S 5]{moore2021some}, and we are grateful to Boban Veli\v ckovi\'c for pointing out that $\OCAsharp$ is probably a consequence of $\OCAT$. 

\begin{theorem}$\OCAT$ implies $\OCAsharp$. 
\end{theorem}\label{T.OCAs}

\begin{proof}
Define $Y=\twoN\times \prod_j \cV_j\times X$. Since each $\cV_j$ is countable, $\prod_j \cV_j$ is naturally homeomorphic to the Baire space and $Y$ has a natural separable metrizable topology. Define a subset $K_0$ of $[Y]^2$ by letting $\{(a,\mu ,x),(b,\nu,y)\}\in K_0$ if
\begin{enumerate}
\item [($K_0$)] \quad $a\neq b$, $x\neq y$, $\mu(\Delta(a,b))=\nu(\Delta(a,b))$, and $\{x,y\}\in \mu(\Delta(a,b))$. 
\end{enumerate}
Since each element of $\bigcup_j \cV_j$ is symmetric, $K_0$ is a symmetric subset of $Y^2$, clearly disjoint from the diagonal. It is evidently an open subset of $[Y]^2$ in its natural topology.

Assume that $H\subseteq Y$ is uncountable and $K_0$-homogeneous. We will prove that the alternative \eqref{OCAsharp.2} of $\OCAsharp$ holds. 
Since $H$ is $K_0$-homogeneous, if $(a,\mu,x)$ and $(b,\nu,y)$ are distinct elements of $H$ then $a\neq b$ and $x\neq y$. Therefore the set 
\[
Z=\{a\mid (a,\mu,x)\in H\text{ for some }\mu,x\}
\] 
is uncountable and $f(a)=x$ if $(a,\mu,x)\in H$ for some $\mu$ defines an injection from $Z$ into $X$. 
Define $\rho\colon \Delta(Z)\to \bigcup_j \cV_j$ as follows. For $s\in \Delta(Z)$ choose $(a,\mu,x)\in H$ with $a(|s|) = 0$ such that there exists some $(b,\nu,y) \in H$ with $a\wedge b = s$ and let $\rho(s)=\mu(|s|)$. To see that $\rho(s)$ does not depend on the choice of $(a,\mu,x)$, pick $(a',\mu',x')\in H$ with $a'(|s|)=0$ such that there exists some $(b',\nu',y') \in H$ with $a'\wedge b' = s$. Then necessarily $a' \wedge b = s$ and since both $\{(a,\mu,x),(b,\nu,y)\}\in K_0$ and $\{(a',\mu',x'),(b,\nu,y)\}\in K_0$ we have $\mu(|s|)=\nu(|s|) = \mu'(|s|)$, as required. This shows $\rho$ is well defined. Further, $\rho(s)=\mu(|s|)\in\mathcal V_{|s|}$ by construction. 
Lastly, if $a$ and $b$ are distinct members of $Z$ then there are unique $(a,\mu,x)$ and $(b,\nu,y)$ in $H$ such that $f(a)=x$, $f(b)=y$, and $\{x,y\}\in \rho(a\wedge b)=\mu(\Delta(a,b))$, and therefore we have the alternative \eqref{OCAsharp.2} of $\OCAsharp$.

Now suppose that $Y$ has no uncountable $K_0$-homogeneous subsets. By $\OCAT$ it can be covered by the union of sets $Y_k$, for $k\in \bbN$, such that $[Y_k]^2\cap K_0=\emptyset$ for all~$k$. Because $K_0$ is an open subset of $[Y]^2$, the property $[Y_k]^2\cap K_0=\emptyset$ is preserved by replacing $Y_k$ by its closure and therefore we can and will assume that for each $k\in \bbN$, $Y_k$ is a closed subset of $Y$.
We will infer that alternative \eqref{OCAsharp.1} of $\OCAsharp$ holds. By the Baire Category Theorem, for each $x\in X$ there exists $k \in \bbN$ such that the closed set
\[
Z_{k,x}=\{(a,\mu)\in \twoN\times \prod_j \cV_j\mid (a,\mu,x)\in Y_k\}
\] 
is nonmeager in $\twoN\times \prod_j \cV_j$. We can therefore pick for each $x \in X$ some $k=k_x$, $m=m_x$, $s_{x}\in \{0,1\}^m$, and $t_{x}\in \prod_{j<m} \cV_j$ such that $[s_{x}]\times [t_{x}] \subseteq Z_{k,x}$. 

Fix $m,k$ in $\bbN$, $s\in \{0,1\}^m$, and $t\in \prod_{j<m} \cV_j$. Let 
\[
X_{k,s,t}=\{x\in X\mid k_x=k, s_{x}=s,\text{ and } t_{x}=t\}. 
\]
We will prove that $[X_{k,s,t}]^2$ is disjoint from $\bigcup\cV_{|s|}$. Towards this, let $m=|s|$ and fix distinct $x$ and $y$ in $X_{k,s,t}$. Let $V\in \cV_m$. In order to prove that $\{x,y\} \notin V$, we additionally fix some $a,b \in [s]$ and $\mu,\nu \in [t]$ 
such that $\Delta(a,b)=m$ and $\mu(m)=\nu(m)=V$. Then $(a,\mu)\in Z_{k,x}$ and $(b,\nu)\in Z_{k,y}$, so since $[X_n]^2$ is disjoint from $K_0$, yet $\mu(\Delta(a,b))=\nu(\Delta(a,b))= V$, 
we necessarily have $\{x,y\}\notin V$. Since $V\in \cV_m$ was arbitrary, this implies that $\{x,y\}\notin \bigcup\cV_m$. Since $x$ and $y$ were arbitrary, this proves that $[X_{k,s,t}]^2$ is disjoint from $\bigcup\cV_m$.

To complete the proof, it remains to re-enumerate the family $\{X_{k,s,t}\}$ as $\tilde X_n$, for $n\in \bbN$, so that $[\tilde X_n]^2\cap \bigcup\cV_n=\emptyset$ for all $n$. Since the sets $\bigcup \cV_n$ form a decreasing sequence, we only need to make sure that $n$ such that $\tilde X_n=X_{k,s,t}$ is not smaller than $|s| = |t|$. Since we are allowed to have $\tilde X_j=\emptyset$ for infinitely many $j$, this is straightforward. 
\end{proof} 

\subsection{A uniformization principle for families of partial selectors}\label{S.partial.selectors}
In this subsection we study families of local liftings of a given homomorphism. In \S\ref{S.Uniformization} we will prove, using $\OCAT$, that they can be uniformized. 

\begin{definition}\label{Def.PartialSelector}
A \emph{partial selector} for a family $\cY_n$, for $n\in \bbN$, is a pair $(g,D(g))$ where $D(g)\subseteq \bbN$ and $g\in \prod_{n\in D(g)} \cY_n$. A \emph{family of partial selectors} for a family $\cY_n$, for $n\in \bbN$ is any 
\begin{equation}\label{eq.cF}
\textstyle\cF\subseteq \{(g,D(g))\mid D(g)\subseteq \bbN, g\in\prod_{n\in D(g)}\cY_n \}. 
\end{equation}
If $\cY_n$ are topological spaces we then topologize the space of partial selectors as follows. Fix a point $0_n\in \cY_n$, and identify $(g,D(g))$ with $\hat g\in \prod_n \cY_n$ which agrees with $g$ on $D(g)$ and satisfies $g(n)=0_n$ for $n\in \bbN\setminus D(g)$. Via this identification, a set $\cF$ of partial selectors is identified with a subset of $\prod_n\cY_n$ and equipped with the subspace topology. 
\end{definition}

In our applications all $\cY_n$ will be separable and metrizable, 
hence the topology on $\cF$ will also be separable and metrizable. 
If all spaces $\cY_n$ are Polish then so is the space of partial selectors. 

For $g,h$ in $\cF$ write
\begin{equation}\label{eq.Diff}
\Diff(g,h)=\{n\in D(g)\cap D(h)\mid g(n)\neq h(n)\},
\end{equation}
and recall that, for $Z\subseteq \twoN$, $\Delta(Z)=\{x\wedge y\mid x,y\in Z, x\neq y\}$.

The following notion was introduced in \cite{Ve:OCA} and \cite{Just:WAT} (but see \cite[II.3.8 and II.4]{Sh:PIF}).

\begin{definition} \label{Def.tree-like}
A family $\cA$ of almost disjoint sets of integers is \emph{tree-like} if there is an order $\prec$ on its domain $\calD=\bigcup\cA$ such that $\<\calD,\prec\>$ is a tree of height~$\omega$ and each element of $\cA$ is included in a unique maximal branch of this tree. 
\end{definition} 

Equivalently, $\cA$ is tree-like if there is an injection $f\colon\bbN\to\twolN$ such that the image of every $A\in\cA$ is included in a single branch of the tree $\twolN$ and for different $A$ and $B$ in $\cA$ the corresponding branches are different. The simplest example of a tree-like family is given as follows. 

\begin{example} \label{ex.perfect} 
Suppose that $J(s)$, for $s\in \twolN$, are pairwise disjoint nonempty finite subsets of $\bbN$. For $h\in \twoN$ let 
\[
B(h)=\bigcup_n J(h\rs n). 
\]
Then $B(h)\cap B(h’)=\bigcup_{s\sqsubset h\wedge h’} J(s)$ for $h\neq h’$. Let $\prec$ be the ordering on $\bigcup_s J(s)$ such that $m\prec n$ if and only if $m$ and $n$ belong to the same $J(s)$ and $m<n$, or if $m\in J(s)$, $n\in J(t)$, and $s\sqsubset t$. Extend $\prec$ to a tree ordering on $\bbN$ such that all $n\in \bbN\setminus \bigcup_s J(s)$ are minimal elements. Then each $B(h)$ is a maximal branch of $(\bbN,\prec)$. 

A tree-like family of the form $\{B(h)\mid h\in \twoN\}$ is said to be \emph{perfect} (because it is a perfect subset of $\cPN$). 
\end{example}

For $s\in \twolN$ let $|s|$ denote its length (equivalently, its domain). 
Case \eqref{3.Unif} of Proposition~\ref{P.Unif} below mirrors one of the main technical contributions in \cite{Fa:AQ}, Lemma 3.13.5, where MA had been used to produce an almost disjoint family $\cA$ and, for every $A\in \cA$, an uncountable set of functions that pairwise disagree on~$A$. We obtain analogous conclusions using $\OCAT$ (under the guise of $\OCAsharp$) only. 

\begin{proposition}\label{P.Unif} 
Assume $\OCAT$. If $\cF$ is a family of partial selectors for a family~$\cY_n$, for $n\in \bbN$, of separable metrizable topological spaces, then one of the following possibilities holds. 
\begin{enumerate}
\item \label{1.Unif} There are $\cF_n$, for $n\in \bbN$, such that $\cF=\bigcup_n \cF_n$, and for all $n$ and all $g,h$ in $\cF_n$ we have $|\Diff(g,h)|\leq n$. 
\item \label{3.Unif} There are a perfect tree-like almost disjoint family $\cA$, an uncountable $Z\subseteq \twoN$, and $f\colon Z\to \cF$ such that for every $A\in \cA$ and all distinct $x,y$ in $Z$ we have 
\[
(\Diff(f(x),f(y))\cap A)\setminus \Delta(x,y)\neq \emptyset.
\] 
\pushcounter
\end{enumerate} 
\end{proposition}

The second part of the following will be applied in situations given by the second alternative of $\OCAsharp$, with $D=\Delta(Z)$ for a carefully chosen uncountable subset $Z$ of the Cantor set. 

\begin{lemma} \label{L.trivial} 
\begin{enumerate}
\item \label{1.trivial} Suppose that $S(n,i)$, for $n\in \bbN$ and $i<2^n$, are finite sets such that $|S(n,i)|\geq n+(4^{n+1}-1)/3$ for all $n$ and $i$. Then there are pairwise disjoint sets $F(n,i)\subseteq S(n,i)\setminus n$ such that $|F(n,i)|=2^n$ for all $n$ and $i$. 
\item \label{2.trivial} If $A(s)\in \Fin$ for $s\in D \subseteq \twolN$ satisfy $|A(s)|\geq |s|+(4^{|s|+1}-1)/3$ for all $s \in D$, then there are pairwise disjoint $B(s)\subseteq A(s)\setminus |s|$ for $s\in D$ such that $|B(s)| = 2^{|s|}$ for all $s \in D$.
\end{enumerate}
\end{lemma}

\begin{proof}
 \eqref{1.trivial}: If $S(n,i)$ are as described, a fairly dumb algorithm produces the sets $F(n,i)$. Let $F(0,0)$ be any element of $S(0,0)$. Assume that $F(k,i)$ as required had been chosen for all $k<n$ and $i<2^k$. Then $I=\bigcup_{k<n,i<2^k} F(k,i)\cup n$ has cardinality not greater than $n+\sum_{k<n,i<2^k} 2^k=n+(4^{n}-1)/3$. Therefore $S(n,i)\setminus I$ has cardinality at least~$4^n$ for each $i<n$. We can now choose pairwise disjoint $F(n,i)\subseteq S(n,i)\setminus I $ of cardinality $2^n$ each recursively in~$i$. 

\eqref{2.trivial}: 
Enumerate the $n$th level of $\twolN$ as $s(n,i)$, for $i<2^n$. Define the sets $S(n,i)$, for $n\in \bbN$ and $i<2^n$ by setting $S(n,i) = A(s(n,i))$ when $s(n,i) \in D$ and by choosing $S(n,i)$ to be an arbitrary set of cardinality $n+(4^{n+1}-1)/3$ when $s(n,i) \notin D$. By part \eqref{1.trivial} of this lemma, there are pairwise disjoint $F(n,i)\subseteq S(n,i)\setminus n$ for $n\in\bbN$ and $i<2^n$, of cardinality $2^{n}$. Then by defining for every $s \in D$ the set $B(s) = F(n,i)$, where $n$ and $i<2^n$ are such that $s = s(n,i)$, we find that the sets $B(s) \subseteq A(s)\setminus |s|$ are as required.
\end{proof}

\begin{proof}[Proof of Proposition~\ref{P.Unif}]
Fix $l$. Since $\cY_l$ is second-countable, let $V(l,m,i)$, for $m\in \bbN$ and $i=0,1$, be basic open sets such that $V(l,m,0)$ and $V(l,m,1)$ are disjoint and for every pair of distinct points $y_0,y_1$ in $\cY_l$ there is $m$ such that $y_i\in V(l,m,i)$ for $i=0,1$. 

For a fixed $n$, let 
\[
\langle S(n,m),u^n_{m}\rangle, \text{ for $m\in \bbN$}, 
\] 
enumerate all pairs such that $S(n,m)\subseteq \bbN$, with $|S(n,m)|=n+(4^{n+1}-1)/3$ and $u^n_{m}\colon S(n,m)\to \bbN$. Let, for all $m$ and $i=0,1$, 
\begin{multline*}
U^n_{m,i}=\{g\in \cF\mid S(n,m)\subseteq D(g),
 g(l)\in V(l,u^n_m(l),i) \text{ for all }l\in S(n,m)\}. 
\end{multline*}
We claim that 
\begin{equation}\label{eq.4n-1/3}
 \{(g_0,g_1)\in \cF^2\mid |\Diff(g_0,g_1)|\geq n+(4^{n+1}-1)/3\} =\bigcup_m U^n_{m,0}\times U^n_{m,1}. 
\end{equation} 
To prove the converse inclusion, note that if $g_i\in U^n_{m,i}$ for some $m$ and $i=0,1$, then $S(n,m)\subseteq D(g_0)\cap D(g_1)$ and for every $l\in S(n,m)$ we have $g_0(l)\neq g_1(l)$, therefore $\Diff(g_0,g_1)\supseteq S(n,m)$. 

To prove the direct inclusion, fix two elements $g_0$ and $g_1$ of $\cF$ which satisfy $|\Diff(g_0,g_1)|\geq n+(4^{n+1}-1)/3$. Then there is $m$ such that $S(n,m)\subseteq \Diff(g_0,g_1)$, $|S(n,m)|=n+(4^{n+1}-1)/3$, $g_i(l)\in V(l,u^n_m(l),i)$ for $i=0,1$ and all $l\in S(n,m)$. 

Let $W^n_m=U^n_{m,0}\times U^n_{m,1}\cup U^n_{m,1}\times U^n_{m,0}$ and let $\cW^n=\{W^n_m\mid m\in \bbN\}$. By $\OCAT$ and Theorem~\ref{T.OCAs}, $\OCAsharp$ holds. 
Assume that the alternative \eqref{OCAsharp.1} of $\OCAsharp$ applied to $X=\cF$ and the family $\cW^n$, for $n\in \bbN$, holds. Fix a partition $X=\bigcup_n X_n$ such that we have $[X_n]^2\cap \bigcup\mathcal W^n=\emptyset$ for all $n$. Then with $\cF_n=X_n$, for all $g,h$ in $\cF_n$ we have $\Diff(g,h)< n + (4^{n+1}-1)/3$, and therefore after re-enumerating the sets $\cF_n$ and interleaving empty sets, the alternative \eqref{1.Unif} of Proposition~\ref{P.Unif} holds. 

It will therefore suffice to prove that the alternative \eqref{OCAsharp.2} of $\OCAsharp$ implies \eqref{3.Unif} of Proposition~\ref{P.Unif}. By \eqref{OCAsharp.2} of $\OCAsharp$, there are $Z\subseteq \twoN$, $f\colon Z\to \cF$, and for every $s\in \Delta(Z)$ there is $m(s)\in \bbN$ such that if $a_0,a_1$ are distinct elements of $Z$ with $s=a_0 \wedge a_1$ and $j=|s|$, then 
\[
\{f(a_0),f(a_1)\}\in W^j_{m(s)}. 
\] 
This means (after swapping $a_0$ with $a_1$ if needed) that $f(a_i)(l)\in V(j,u^j_{m(s)}(l),i)$ for all $l\in S(j,m(s))$ and $i=0,1$. For $s\in \Delta(Z)$ let $A(s)=S(|s|,m(s))$. Then \eqref{eq.4n-1/3} implies $\Diff(f(a_0), f(a_1))\supseteq A(s)$, and $|A(s)|= |s|+(4^{|s|+1}-1)/3$. We may assume that $Z$ has no isolated points by removing all basic open sets $U$ such that $Z\cap U$ is countable from $Z$. Furthermore, \eqref{2.trivial} of Lemma~\ref{L.trivial} implies that there are pairwise disjoint sets $B(s)\subseteq A(s)\setminus |s|$ such that $|B(s)| = 2^{|s|}$, for $s\in \Delta(Z)$. To recap, we have the following. 
\begin{enumerate}\popcounter
\item \label{2.Unif} There are an uncountable $Z\subseteq \twoN$, $f\colon Z\to \cF$, and pairwise disjoint $B(s)\subset \bbN$ for $s\in \Delta(Z)$ such that $|B(s)|= 2^{|s|}$ and for all $x\neq y$ in~$Z$ we have $\Diff(f(x),f(y)) \setminus \Delta(x,y) \supseteq B(x\wedge y)$. 
\end{enumerate}
To prove that this implies \eqref{3.Unif}, enumerate each $B(s)$ as $k(s,t)$, for $t\in \{0,1\}^{|s|}$. For $h\in \twoN$ let 
\[
B_h=\{k(s,h\rs |s|)\mid s\in \Delta(Z)\}.
\]
For $h\neq h'$ in $\twoN$ we have $B_h\cap B_{h'}\subseteq \bigcup_{|s|<\Delta(h.h')} B(s)$, and therefore these sets form a perfect tree-like almost disjoint family $\cA$. Since for distinct $x$ and $y$ in~$Z$ we have $\Diff(f(x),f(y)) \setminus \Delta(x,y) \supseteq B(x\wedge y)$ and $B(x\wedge y)\cap B_h\neq \emptyset$, \eqref{3.Unif} follows. 
\end{proof}

\section{Automorphisms of $\cP(\bbN)/\Fin$}\label{S.Uniformization.Fin}

The following is the first part of Theorem~\ref{thm:OCAPomega} as stated in the introduction.
\begin{theorem} \label{T.OCA.Fin}
 Assume $\OCAT$. Then all automorphisms of $\cP(\bbN)/\Fin$ are trivial. 
\end{theorem}

We will prove a stronger result along the lines of the $\OCA$ lifting theorem of \cite{Fa:AQ} (see Theorem~\ref{T.OCAsharp-Fin}, which is the second part of Theorem~\ref{thm:OCAPomega} as stated in the introduction). The proof of this theorem occupies this section.

Let us isolate two important ideals.
\begin{definition}\label{defin:ideals}
Let $\Phi\colon\mathcal P(\bbN)/\Fin\to\cP(\bbN)/\Fin$ be a homomorphism. Define 
\[
\Jcont(\Phi):=\{A\subseteq\bbN\mid\Phi\restriction \cP(A)/\Fin\text{ has a continuous lift}\}. 
\]
By a $\sigma$-Borel lifting, we mean a lifting whose graph can be covered by countably many Borel functions. 
Let 
\[
\cJ_\sigma(\Phi)=\{A\subseteq \bbN\mid \Phi\restriction\mathcal P(A)/\Fin\text{ has a }\sigma\text{-Borel lifting}\}.
\]
When $\Phi$ is clear from the context, we just refer to $\Jcont$ and $\cJ_\sigma$.
\end{definition}

The following relies on Proposition~\ref{P.Jsigma-to-Jcont}, which is included in the Appendix since it is a well-known fact yet perhaps not so well documented.

\begin{proposition} \label{P.Jcont.nonmeager}
Assume $\OCAT$. If $\Phi\colon \cP(\bbN)/\Fin\to \cP(\bbN)/\Fin$ is a homomorphism, then $\Jcont$ intersects every perfect tree-like almost disjoint family nontrivially. In particular, it is nonmeager. 
\end{proposition}

\begin{proof}
First, given a perfect tree-like almost disjoint family $\cA_0$ we construct another perfect tree-like almost disjoint family $\cA$ such that each element of $\cA$ includes infinitely many elements of $\cA_0$. The most natural way to do so is to assure that every element of $\cA$ is a union of a perfect subset of $\cA_0$. 

Fix a perfect tree-like almost disjoint family $\cA_0$. There are pairwise disjoint finite sets $J(s)$, for $s\in \twolN$, such that $\cA_0$ is the set of all $B(h)=\bigcup_n J(h\rs n)$, for $h\in \twoN$. 
For $s\in \twolN$ with length $n$, let 
\[
\tilde J(s)=\bigcup\{J(t)\mid 2n \leq |t|<2(n+1), (\forall j< n) t(2j+1)=s(j)\}.
\]
This is a family of finite, pairwise disjoint, sets. 
Let $\cA$ be the family of all 
\[
\tilde B(f)=\bigcup_n \tilde J(f\rs n)
\]
for $f\in \twoN$. Then $\tilde B(f)=\bigcup\{B(h)\mid h(2j+1)=f(j)$ for all $j\}$, and $\cA$ is as required. 

Next, we prove that the ideal $\cJ_\sigma$ intersects $\cA$ nontrivially. The proof of this is 
virtually identical to that of \cite[Lemma~2.2]{Ve:OCA}, where it was proven that $\OCAT$ implies that all but countably many of the sets in $\cA$ belong to $\cJ_\sigma$. Note, in \cite[Lemma~2.2]{Ve:OCA} the assumption that $\Phi$---denoted $\varphi$ in \cite{Ve:OCA}---is an automorphism isn’t used until after $\Phi$ is proven to have a $\sigma$-Borel-measurable lifting on $\cP(B)$ for some set~$B$ in the given tree-like almost disjoint family, at which point \cite[Theorem~1.2]{Ve:OCA} is invoked. As this theorem relies on the assumption that $\Phi$ is an automorphism, at this point we depart from the proof of \cite[Lemma~2.2]{Ve:OCA}. 
Consider $f$ such that $\tilde B(f)\in \cJ_\sigma$. Then $\tilde B(f)$ contains a union of infinitely many elements of $\cA_0$, and Proposition~\ref{P.Jsigma-to-Jcont} proven in the Appendix implies that all but finitely many of these elements of $\cA_0$ belong to $\Jcont$. Since $\cA_0$ was arbitrary, this proves the first part.

It remains to prove that $\Jcont$ is nonmeager.
Since the constant map sending $\mathcal P(\bbN)$ to the empty set is continuous and lifts $\Phi$ on $\Fin$, $\Jcont$ contains $\Fin$. We work by contradiction, and assume that $\Jcont$ is meager. By the well-known characterization of nonmeager ideals of Talagrand and Jalali-Naini (\cite[Theorem~3.10.1]{Fa:AQ}), there are pairwise disjoint finite intervals of~$\bbN$, $F_j$, for $j\in \bbN$, such that $\bigcup_{j\in A} F_j\in \Jcont$ if and only if $A$ is finite. 
Fix any perfect tree-like almost disjoint family $\cA_0$ and let 
\[
\cA=\{\bigcup_{i\in A} F_i\mid A\in \cA_0\}.
\]
This is a natural perfect tree-like almost disjoint family 
associated to~$\cA_0$ and therefore $\Jcont$ intersects it; contradiction. (Succinctly said, this is the image of $\cA_0$ under the endomorphism of $\cP(\bbN)/\Fin$ associated with the Rudin--Blass reduction of $\Fin$ to itself corresponding to the function that collapses $F_i$ to $i$ for all $i\in \bbN$.) 
\end{proof}

Fix a lifting $\Phi_*$ of $\Phi$. A continuous homomorphism lifting an endomorphism of $\mathcal P(\bbN)/\Fin$ on some set $A$ is \emph{completely additive} (see \cite[Corollary 3.8.2]{Fa:AQ}). This means that for each $A$ in $\Jcont$ the lifting of the restriction of $\Phi$ to $\cP(A)/\Fin$ is of the form $B\mapsto h_A^{-1}(B)$ for a function 
\[
h_A\colon \Phi_*(A)\to A.
\]
It will be convenient to `look forward' and consider the functions $g_A\colon A\to \Fin$ defined by 
\[
g_A(n)=h_A^{-1}(n)
\]
instead. With this definition, the function $\Phi_{h_A}\colon \cP(\bbN)\to\cP(\bbN)$ defined by $X\mapsto \bigcup_{n\in X}g_A(n)$ is a completely additive lifting of the restriction of $\Phi$ to $\cP(A)$, that is, for all $A\in\Jcont$ and $B\subseteq A$ we have that
\[
\Phi_{h_A}(B)=^*\Phi_*(B).
\] 
We fix these functions to start with, but reserve the right to modify them as convenient. 
The following is a reformulation of the fact that the functions $h_A$ form a coherent family (\cite{Fa:AQ}). 

\begin{lemma}\label{L.hA.coherent}
Fix a homomorphism $\Phi\colon \cP(\bbN)/\Fin\to \cP(\bbN)/\Fin$. Then for all $A$ and $B$ in $\Jcont$ the following holds. 
\begin{enumerate}
\item\label{L.Coherent1} The set of $m\in A\cap B$ such that $g_A(m)\neq g_B(m)$ is finite.
\item\label{L.Coherent2} The set of pairs $(m,n)$ such that $m\in A$, $n\in B$, $m\neq n$, and $g_A(m)\cap g_B(n)\neq \emptyset$ is finite. 
\end{enumerate}
\end{lemma}

\begin{proof}
We prove \eqref{L.Coherent1} by contradiction. Assume that for some $A$ and $B$ in $\Jcont$ the set $C_0=\{n\mid g_A(n)\neq g_B(n)\}$ is infinite. 
 
Since both $A$ and $B$ belong to $\Jcont$, then $\mathcal P(A)=\{C\mid \Phi_{h_A}(C)=^*\Phi_*(C)\}$ and $\mathcal P(B)=\{C\mid\Phi_{h_B}(C)=^*\Phi_*(C)\}$, hence 
\[
\Phi_{h_A}(C_0)=^*\Phi_*(C_0)=^*\Phi_{h_B}(C_0). 
\]
All of the sets $g_A(n)$ and $g_B(n)$ are finite. We can therefore recursively find an infinite subset $C'$ of $C_0$ such that $g_A(n)\cap g_B(m)=\emptyset$ for all distinct $n$ and $m$ in $C'$. In particular $\bigcup_{n\in C'}g_A(n)=\Phi_{h_A}(C')$ and $\bigcup_{n\in C'}g_B(n)=\Phi_{h_B}(C')$ have infinite symmetric difference. This is a contradiction to the fact that $\Phi_{h_A}(C')=^*\Phi_*(C')=^*\Phi_{h_B}(C')$.
 
We now prove \eqref{L.Coherent2}. Assume that $A$ and $B$ are in $\Jcont$ and such that the set of pairs $(m,n)$ such that $m\in A$, $n\in B$, $m\neq n$, and $g_A(m)\cap g_B(n)\neq \emptyset$ is infinite. Since $g_A(m)$ is finite for all $m$, we can choose pairs $m_i,n_i$ so that $g_A(m_i)\cap g_B(n_i)$ is nonempty, $m_i\neq n_j$ for all $i$ and $j$, and both $m_i\neq m_j$ and $n_i\neq n_j$ for all distinct $i$ and $j$. Let $C=\{n_i\}$ and $D=\{m_i\}$. Then $C$ and $D$ are disjoint subsets of $A$ and $B$ respectively, but $\Phi_{h_A}(C)=^*\Phi_*(C)$ and $\Phi_{h_B}(D)=^*\Phi_*(D)$ have infinite intersection. This is a contradiction.
\end{proof}

\begin{definition}\label{Def.L0}
For $D\subseteq \N$ define a partition 
\[
[\Jcont]^2=L_0(D)\cup L_1(D)
\]
by letting $\{A,B\}$ in $L_0(D)$ if and only if there exists $n\in A\cap B\cap D$ such that 
\[
g_A(n)\neq g_B(n)
\]
or there exist distinct $m\in A\cap D$ and $n\in B\cap D$ such that 
\[
g_A(m)\cap g_B(n)\neq \emptyset. 
\]
We write $L_0=L_0(\bbN)$ and $L_1=L_1(\bbN)$. 
\end{definition}

By identifying $A$ with the graph of~$g_A$, we see that $\OCAT$ applies to the partition $[\Jcont]^2=L_0(D)\cup L_1(D)$.

\begin{lemma} \label{L.D.Uniformization} 
If $\Phi\colon \cP(\bbN)/\Fin\to \cP(\bbN)/\Fin$ is a homomorphism then for every $D$ in $ \Jcont$, the ideal $ \Jcont$ is $\sigma$-$L_1(D)$-homogeneous. 
\end{lemma}

\begin{proof} 
Assume $D\in \Jcont$ with $g_D\colon D\to \Fin$ as a witness (the sets in the range of $g_D$ are pairwise disjoint). For $n\in \N$ let 
\[
\X_n=\{A\in \Jcont\mid g_A(m)=g_D(m)
\text{ for all }m\in (A\cap D)\setminus (n+1)\}. 
\]
If $A$ and $B$ belong to $\X_n$ and in addition $g_A$ and $g_B$ agree on $D\cap (n+1)$, then $\{A,B\}\in L_1(D)$. Since there are only countably many functions from $ (n+1)$ into~$\bbN$, each $\X_n$ can be covered by countably many $L_1(D)$-homogeneous sets. Then $\Jcont\subseteq \bigcup_n \cX_n$ by Lemma~\ref{L.hA.coherent}, and the conclusion follows. 
\end{proof}

Modulo introducing some terminology, we are ready to present a sharpening of Theorem~\ref{T.OCA.Fin} which is a version of the $\OCA$ lifting theorem of \cite{Fa:AQ}. The conclusions of the two theorems are identical, but while the former uses $\OCAT$ and $\MA$, the latter uses only $\OCAT$. 
If $\Phi\colon \cPN\to \cPN/\Fin$ is a homomorphism and $A\subseteq \bbN$, consider $\Phi_A\colon \cP(\bbN)\to \cP(A)/\Fin$ defined by ($\Phi_*$ is any lifting of $\Phi$)
\begin{equation}
\Phi_A(B)=[\Phi_*(B)\cap A]_{\Fin}.
\end{equation}
Thus $\Phi$ decomposes as $\Phi_A\oplus \Phi_{\bbN\setminus A}$, 
\begin{center}
\begin{tikzpicture}
\matrix[row sep=1cm,column sep=1cm]
{
& \node (PA) {$\cP(A)/\Fin$}; \\
\node (PN) {$\cP(\N)$}; 
& & \node (PN1) {$\cP(\bbN)/\Fin$};\\
& \node (PB) {$\cP(\bbN\setminus A)/\Fin$}; \\ 
};
\draw (PN) edge [->] node [above] {$\Phi_A$} (PA); 
\draw (PN) edge [->] node [below] {$\Phi_{\bbN\setminus A}$} (PB); 
\draw (PN) edge [->] node [below] {$\Phi$} (PN1); 
\draw (PA) edge [->] node [above] {$i_1$} (PN1); 
\draw (PB) edge [->] node [below] {$i_2$} (PN1); 
\end{tikzpicture}
\end{center}

A function between Polish spaces is called C-measurable if it is measurable with respect to the $\sigma$-algebra generated by analytic sets\footnote{Kechris in \cite[\S29.D]{Ke:Classical} calls such functions $\sigma(\Sigma^1_1)$-measurable, while C-measurability denotes a weaker condition. Nevertheless, we stick to this terminology as it is commonly used in the area of study of liftings (e.g., \cite[\S17]{Fa:STCstar}, \cite{mckenney2018forcing}, or \cite{Fa:All}).}. In the following proof, we use the Jankov--von Neumann uniformization theorem, stated for convenience in Appendix~\ref{S.sigma-Borel}.

\begin{theorem} \label{T.OCAsharp-Fin}
Assume $\OCAT$. If $\Phi\colon \cP(\N)/\Fin\to \cP(\N)/\Fin$ is a homomorphism, then $\Phi$ can be decomposed as $\Phi_1\oplus\Phi_2$ so that $\Phi_1$ has a completely additive lifting, while the kernel of $\Phi_2$ is nonmeager.
\end{theorem}

\begin{proof}
Recall that for every $A\in \Jcont$ we have fixed $g_A\colon A\to \Fin$ such that $B\mapsto \bigcup_{n\in B} g_A(n)$ is a lifting of $\Phi$ on $\cP(A)$. Let
\[
\cF=\{(g_A,A)\mid A\in \Jcont\}. 
\]
This is a family of partial selectors, with $\cY_n=\Fin$ for all $n$ (see \S\ref{S.partial.selectors}; this proof makes heavy use of terminology and results from this section). 
$\OCAT$ and Proposition~\ref{P.Unif} together give us two alternatives. 

If \eqref{3.Unif} of Proposition~\ref{P.Unif} applies, then there are a perfect tree-like almost disjoint family~$\cA$, an uncountable $Z\subseteq \twoN$, and a function $f\colon Z\to \cF$ such that for every $A\in \cA$ and all distinct $x,y$ in $Z$ we have 
\[
(\Diff(f(x),f(y))\cap A)\setminus \Delta(x,y)\neq \emptyset.
\] 
By Proposition~\ref{P.Jcont.nonmeager} and $\OCAT$, $\Jcont$ intersects $\cA$ nontrivially. 
If $A\in \cA\cap \Jcont$ then Lemma~\ref{L.D.Uniformization} implies that $\Jcont$ is $\sigma$-$L_1(A)$-homogeneous. However, if $x$ and $y$ are distinct elements of $Z$ with $f(x)=(g_{A_1},A_1)$ and $f(y)=(g_{A_2},A_2)$, then the existence of an element of $\Diff(g_{A_1},g_{A_2})\cap A$ gives that $\{A_1,A_2\}\in L_0(A)$. This would give that $f[Z]$ is an uncountable $L_0(A)$-homogeneous subset of the $\sigma$-$L_1(A)$-homogeneous set $\Jcont$, which is a contradiction.

Because of this and $\OCAT$, \eqref{1.Unif} of Proposition~\ref{P.Unif} holds. Fix $\cF_n$, for $n\in \bbN$, such that $\cF=\bigcup_n \cF_n$, and for all $n$ and all $(g_A,A)$ and $(g_B,B)$ in $\cF_n$ we have $|\Diff(g_A,g_B)|\leq n$. 
Since $\Jcont$ is nonmeager, there is $n$ such that 
\[
\cX_n=\{A\mid (g_A,A)\in \cF_n\}
\] 
is nonmeager. 
Let
\[
\cG=\overline{\{(g,B)\mid \exists (g_A,A)\in \cF_n (B\subseteq A, g=g_A\rs B)\}}. 
\]
The closure is taken in the space of partial selectors as in Definition~\ref{Def.PartialSelector}. Since each~$\cY_n$ is the discrete space $\Fin$, this is a Polish topology. Therefore the set 
\[
\cW=\{B\mid (\exists g) (g,B)\in \cG\}
\]
is, being the projection of a closed set, an analytic subset of $\cP(\bbN)$. 
It is also hereditary and (because it includes $\cX_n$) nonmeager. 
Note also that $\cG$ has the critical property that for all $(g_0,B_0)$ and $(g_1,B_1)$ in $\cG$ we have $|\Diff(g_0,g_1)|\leq n$.

Since $\cW$ is nonmeager and analytic, there is a basic open subset $U$ of $\cP(\bbN)$ such that $U\setminus \cW$ is meager (see e.g. \cite[Proposition 8.26]{Ke:Classical}). It is of the form $U=\{A\mid A\cap J_0=s_0\}$ for some finite interval $J_0$ and $s_0\subseteq J_0$. As $\cW$ is hereditary, we may assume that $s_0=\emptyset$.

Since $\cW$ is relatively comeager in $\cP(\bbN\setminus J_0)$, 
there is a partition of $\bbN\setminus J_0$ into finite intervals $J_i$, for $i\geq 1$, and $s_i\subseteq J_i$ for all $i$ such that 
\[
\{A\subseteq \bbN\setminus J_0\mid (\exists^\infty i) A\cap J_i=s_i\}\subseteq \cW.
\]
Let $S=\bigcup s_i$. As $s_i\subseteq J_i$ and $\{J_i\}$ partitions $\mathbb N\setminus J_0$, $S\in\mathcal W$.
 
Let $A_\even=\bigcup_{i\geq 1, i\ \even} J_i$, $A_\odd=\bigcup_{i\ \text{odd}} J_i$, $S_\even=S\cap A_\even$, and $S_\odd=S\cap A_\odd$, we have that for every $B\subseteq \bbN$, both $B_\even=(B\cap A_\even )\cup S_\odd$ and $B_\odd=(B\cap A_\odd)\cup S_\even$ belong to $\cW$. Using this notation, let 
\begin{equation*}
	\begin{split}
\cV= \{ &(g,B)\mid B\subseteq \bbN, g:B\to\Fin,\\ &\text{and some $h_\even$ and $h_\odd$ satisfy } (h_\even,B_\even)\in\cG, (h_\odd,B_\odd)\in\cG,\\& B \subseteq B_\even \cup B_\odd, \Diff(h_\even,g) = \emptyset \text{ and }|\Diff(h_\odd,g)|\leq n \}.
\end{split}
\end{equation*}
\begin{claim}\label{claimBG}
For every $B\subseteq \bbN$ there is $g$ such that $(g,B)\in \cV$ and every such $g$ satisfies $|\Diff(g,g_C)|\leq 3n$ for every $C\in\cX_n$.
\end{claim} 
\begin{proof}
Fix $B\subseteq \bbN$. Since both $B_\even$ and $B_\odd$ belong to $\mathcal W$, some $h_\even$ and~$h_\odd$ satisfy $(h_\even,B_\even)\in \cG$ and $(h_\odd,B_\odd)\in \cG$. Let 
\[
g=h_\even\rs (B_\even\cap B) \cup (h_\odd\rs (B\setminus B_\even)).
\] 
 By the homogeneity of the set $\cG$ we have $|\Diff(h_\odd,h_\even)|\leq n$, and therefore $\Diff(h_\odd,g)\subseteq \Diff(h_\odd,h_\even)$ implies $(g,B)\in\cV$.

 To prove the second part of the claim, fix $C\in\mathcal X_n$. Then $(g_C,C)\in \mathcal G$, and therefore $\max(|\Diff(g_C,h_\even)|,|\Diff(g_C,h_\odd)|)\leq n$. Since 
\begin{align*}
\Diff(g,g_C)\subseteq&\Diff(g\rs B_\even,h_\even)\cup \Diff(g\rs B_\odd,h_\odd)\\
&\cup \Diff(h_\even,g_C\rs B_\even)\cup \Diff(h_\odd,g_C\rs B_\odd),
\end{align*} 
and $\Diff(g\rs B_\even,h_\even)=\emptyset$, we have $|\Diff(g,g_C)|\leq 3n$ as required.
\end{proof}

By the Jankov--von Neumann uniformization theorem (Theorem~\ref{JVN}) there is a C-measurable function $\Theta\colon \cP(\bbN)\to \bigcup_{A \subseteq \bbN}\prod_{n \in A} \cY_n$ such that $(\Theta(A),A)\in \cV$ for all~$A$.

We claim that $\Diff(\Theta(A),g_A)$ is finite for every $A\in \Jcont$. Assume otherwise and fix $A\in\Jcont$ for which this set is infinite. Since $\cX_n$ is nonmeager, its hereditary closure $\hat\cX_n=\{B\subseteq \bbN\mid (\exists C\in \cX_n) B\subseteq C\}$ is nonmeager and hereditary and by \cite[Theorem~3.10.1]{Fa:AQ} there is an infinite $B\in \hat\cX_n$ included in $\Diff(\Theta(A),g_A)$. If $C\in \cX_n$ with $B\subseteq C$, then $B\subseteq \Diff(g_A,g_C)\cup \Diff(g_C,\Theta(A))$.
However, each one of these two sets is finite: $\Diff(g_A,g_C)$ is finite because both $A$ and $C$ belong to $\Jcont$ and $|\Diff(g_C,\Theta(A))|\leq 3n$ by Claim~\ref{claimBG}; contradiction.
 
Therefore the function $\tilde\Theta\colon \cP(\bbN)\to \cP(\bbN)$ defined by 
\[
\tilde\Theta(A)=\bigcup\Theta[A]
\]
is a C-measurable lifting of $\Phi$ on $\Jcont$. By \cite[Lemma~3.3.4]{Fa:AQ}, $\Phi$ decomposes as $\Phi_1\oplus\Phi_2$ where $\Phi_1$ has a continuous lifting and the kernel of $\Phi_2$ is nonmeager. By \cite[Theorem~1.6.1]{Fa:AQ}, $\Phi_1$ has a completely additive lifting, as required. 
\end{proof}

The conclusion of Theorem~\ref{T.OCAsharp-Fin} is incompatible with CH. An obvious route towards strengthening it is by requiring $\Phi$ to have a continuous lifting; this strengthening was even named Strong Extension Principle and conjectured to follow from forcing axioms in \cite[\S 4]{Fa:AQ} (the weak Extension Principle is a generalization of the conclusion of Theorem~\ref{T.OCAsharp-Fin}---or rather, its dual restatement in terms of Boolean algebras, see \cite[\S 1.3.1]{Fa:STCstar}---to \v Cech--Stone remainders of locally compact Polish spaces other than $\bbN$). This conjecture is, however, false. 
By \cite{dow2014non}, there is an endomorphism of $\cP(\bbN)/\Fin$ which does not have a lifting which is an endomorphism of $\cP(\bbN)$. 

\section{Proof of Theorem \ref{th:main}}\label{S:proof} 

Each of the two subsections of the present section is centered around a proposition. These two, together, will give the main result of this section:
\begin{theorem}\label{T.OCACR}
Assume $\OCA_T$ and $\MAsl$. Suppose that $\cM_n$ and $\cN_n$, for $n\in \bbN$, are countable sets and 
\[
\Phi\colon \prMF\to \prNF
\]
is a coordinate-respecting function. Then $\Phi$ is of twisted product form.
\end{theorem}
Once Theorem~\ref{T.OCACR} is proved, Theorem~\ref{th:main} follows immediately:
\begin{proof}[Proof of Theorem~\ref{th:main}]
	Suppose that $\cM_n$ and $\cN_n$, for $n\in \bbN$, are countable structures of the same language and 	$\Phi\colon \prMF\to \prNF$	is a coordinate-respecting isomorphism. By $\OCAT$, $\MAsl$, and Theorem~\ref{T.OCACR}, $\Phi$ is of twisted product form. 
	Lemma~\ref{L.Phi.1} gives that $\Phi$ is trivial, and therefore the thesis.
\end{proof}

Let us fix some notation. Suppose that $\cM_n$, $\cN_n$, for $n\in\bbN$, are countable sets. Since $\prod_n\cM_n$ is empty as soon as some $\cM_n$ is, and the product of singletons is a singleton, we will always assume that all sets $\cM_n$ and $\cN_n$ have size at least two. We let $\cM=\prod_n\cM_n/\Fin$ and $\cN=\prod_n\cN_n/\Fin$, while for an infinite $X\in\cP(\bbN)$, we write $\cM\restriction X$ for $\prod_{n\in X}\cM_n/\Fin$, and let $\pi_X\colon \cM\to \cM\restriction X$ be the canonical quotient map. 

Let $ \Phi\colon \prMF\to \prNF$ be a coordinate-respecting function. By definition, we have an isomorphism $\alpha=\alpha^\Phi\colon \cPNF\to \cPNF$ such that for every $X\in \cPNF$ there exists a function $\Phi_X \colon \cM\rs X \to \cN\rs\alpha(X)$ that makes the following diagram commute
\begin{center}
\begin{tikzpicture}
\matrix[row sep=1cm,column sep=1cm]
{
& & \node (M1) {$\cM$}; && &\node (M2) {$\cN$};&\\
& & \node (Q1) {$\cM\rs X$}; &&& \node (Q2) {$\cN\rs\alpha(X)$} ;\\
};
\draw (M1) edge [->] node [above] {$\Phi$} (M2);
\draw (Q1) edge [->] node [above] {$\Phi_X$} (Q2);
\draw (M1) edge [->] node [left] {$\pi_X$} (Q1);
\draw (M2) edge [->] node [right] {$\pi_{\alpha(X)}$} (Q2);
\end{tikzpicture}
\end{center}
By $\OCAT$ and Theorem~\ref{T.OCA.Fin}, $\alpha$ is trivial, hence there is a bijection $f$ between cofinite subsets of $\bbN$ such that $X\mapsto f^{-1}[X]$ lifts $\alpha$. By removing finite sets from $\bbN$ we may assume that $f$ is a permutation of $\bbN$, and by re-indexing the $\cN_n$ we may assume that $f$ is the identity function. 
 
We will show that $ \Phi$ modified in this way is of product form, which will imply that the original $\Phi$ was of twisted product form. 

In Proposition~\ref{lem:IprodccmodFin} we will prove, using $\OCAT$ and $\MAsl$, that the ideal 
\begin{equation}\label{eq.Jprod}
\cJprod=\cJprod( \Phi)=\{A\subseteq \bbN\mid\text{ $\Phi_A$ is of product form}\}
\end{equation}
is nonmeager, thus proving a local version of Theorem~\ref{T.OCACR}. In \S\ref{S.Uniformization} we will use $\OCAT$ to uniformize local liftings associated with elements of $\cJprod$ and obtain a lifting of product form.

\subsection{Local version of Theorem~\ref{T.OCACR}}
 
A subset of a poset $\bbP$ is \emph{linked} if every two conditions in it are compatible. A poset $\bbP$ is said to be \emph{$\sigma$-linked}\index{S@$\sigma$-linked} if $\bbP$ is the union of countably many linked subsets. 

$\MAsl$ asserts that if $\bbP$ is $\sigma$-linked then for every family of $\aleph_1$ dense subsets of $\bbP$ some filter intersects all of them (see e.g., \cite{Ku:Book}). 

In the Proposition~\ref{lem:IprodccmodFin} below and elsewhere, countable sets are considered with respect to the discrete topology hence every product $\prod_n \cM_n$ of such sets is equipped with a natural Polish product topology.

\begin{proposition}\label{lem:IprodccmodFin} 
Assume $\OCAT$ and $\MA(\sigma$-linked). 
 If $\cM_n$ and $\cN_n$, for $n\in \bbN$, are countable sets and 
\[
\Phi\colon \prMF\to \prNF
\]
is a coordinate-respecting function with $\alpha^\Phi = \id$, then the ideal $\cJprod$ intersects every uncountable almost disjoint family. In particular $\cJprod$ is nonmeager.
 \end{proposition}

\begin{proof} 
 We will prove that for every almost disjoint family $\cA$ all but countably many of its elements belong to $\cJprod$ (see \eqref{eq.Jprod}). Using the well-known characterization of nonmeagerness of ideals on $\bbN$ due to Jalali-Naini and Talagrand (\cite[Theorem~3.10.1]{Fa:AQ}) again, this will imply that the ideal $ \cJprod$ is nonmeager. 

Fix $\cA$ and a lifting $\Phi_*$ of $\Phi$. Let $\cX=\{ (a,A) : A\in\mathcal{A}, a \in \prM \}$. Consider the partition $[\cX]^2 = K_0 \sqcup K_1$ defined by $\{(a,A),(b,B)\} \in K_0$ if and only if
\begin{enumerate}
\item $A\neq B$ and 
\item there is $n\in A\cap B$ such that $a(n)=b(n)$ and $\Phi_*(a)(n)\neq\Phi_*(b)(n)$.
\end{enumerate}
Note that $K_0$ is an open subset of $[\cX]^2$ in the separable metrizable topology obtained when $\cX$ is viewed as a subset of $\prM\times \cPN\times \prN$, via the identification $(a,A) \mapsto (a,A,\Phi_*(a))$.
 
We first aim to show that there exists no uncountable $K_0$-homogeneous subset of~$\cX$. Suppose this is not the case, and let $H\subseteq\cX$ be an uncountable $K_0$-homogeneous set. Consider the forcing $\mathbb{P}$ with conditions of the form $ p=(F_p,s^0_p,s^1_p,n_p)$ satisfying
\begin{itemize}
\item $F_p\subset H$ is finite, $n_p\in \bbN$ and $s^0_p,s^1_p\in\prod_{n<n_p}\cM_n$,
\item for every $(a,A),(b,B)\in F_p$ with $(a,A)\neq (b,B)$, $A\cap B\subseteq n_p$,
\item $\forall (a,A)\in F_p$, $\forall n\in A\cap n_p$, $a(n)\in \{ s^0_p(n), s^1_p(n) \}$. 
\end{itemize}
The order relation on $\mathbb{P}$ is given by
\[
q\leq p\Leftrightarrow F_p\subseteq F_q, n_p\leq n_q, s^0_p\subseteq s^0_q \text{ and } s^1_p\subseteq s^1_q.\]
\begin{claim} \label{C.5.3}The poset
$\mathbb{P}$ is $\sigma$-linked. 
\end{claim}
\begin{proof}
To $p\in \bbP$ associate the invariant
\[
\bbI(p)=\< s^0_p, s^1_p, (a\upharpoonright n_p,A\upharpoonright n_p): (a,A)\in F_p\>.
\]
There are only countably many choices for $\mathbb I(p)$. We prove that if $\mathbb I(p)=\mathbb I(q)$, then $p$ and $q$ are compatible. This will conclude the proof of the claim.

Define $F_r=F_p\cup F_q$ and choose $n_r$ large enough to have $A\cap B\subseteq n_r$ for all distinct $(a,A)$ and $(b,B)$ in $F_r$. It remains to define $s^0_r(n)$ and $s^1_r(n)$ for every $n\in n_r\setminus n_p$ so that $r$ is a condition. Any such $r$ will extend both $p$ and $q$. 

Fix $n\in n_r\setminus n_p$. Since $p\in \bbP$, there is at most one $(a,A)\in F_p$ such that $n\in A$. Similarly, there is at most one $(b,B)\in F_q$ such that $n\in B$. If both $(a,A)$ and $(b,B)$ are defined, 
let $s^0_r(n)=a(n)$ and $s^1_r(n)=b(n)$. Otherwise, we may choose one (or both) $s^j_r(n)$ arbitrarily. 

Clearly $r$ defined in this manner belongs to $\bbP$ and extends both $p$ and $q$. 
\end{proof}

Let $\{ (a_\xi, A_\xi): \xi < \omega_1\}$ be an enumeration of $H$ (by having passed to a subset if necessary, we may assume that $H$ has cardinality $\aleph_1$). Consider the sets $E_\xi=\{q\in \bbP\mid (\exists \eta > \xi) (a_\eta, A_\eta) \in F_q\}$, for $\xi < \omega_1$, and $E_n=\{q\in \bbP\mid n_q>n\}$, for $n\in\bbN$. 
Each $E_n$ is dense in $\bbP$, as can be seen from the proof of Claim~\ref{C.5.3}. 
Each one of the $E_\xi$ is a nonempty subset of $\bbP$. Since $\bbP$ has the countable chain condition, some $p\in \bbP$ forces that the generic filter $\dot G$ intersects uncountably many of the $E_\xi$ nontrivially. (To see this, assume otherwise and fix a maximal antichain consisting of conditions that decide the supremum of the set of all $\xi$ such that $\dot G\cap E_\xi\neq \emptyset$. Since the antichain is countable we have $E_\xi=\emptyset$ for all but countably many $\xi$; contradiction.) Since $\xi<\eta$ implies $E_\xi\supseteq E_\eta$, each $E_\xi$ is dense below $p$.

By $\MA(\sigma$-linked) there exists a filter $G \subseteq \bbP$ that intersects all $E_n$ and all $E_\xi$ nontrivially. Therefore $\bigcup_{q\in G} F_q \subseteq H$ is uncountable and the set $\{n_q: q\in G\}$ is infinite. Hence, by shrinking $H$ if necessary we can actually assume that there exist $s^0,s^1\in\prod_n\cM_n$ such that 
\[ 
\forall (a,A)\in H \ \forall n\in A\ a(n)\in\{s^0(n),s^1(n)\}. 
\]
Fix $(a,A)\in H$, and let 
\[
I^0_{(a,A)}=\{ n\in A: a(n)=s^0(n) \}\text{ and }
I^1_{(a,A)}=\{ n\in A: a(n)=s^1(n) \}. 
\]
Note that $A= I^0_{(a,A)}\cup I^1_{(a,A)}$, and 
\[
\pi_{I^i_{(a,A)}}(a)=\pi_{ I^i_{(a,A)}}(s^i)\text{ for }i=0,1.
\]
As $\alpha=\alpha^\Phi$ is the identity, we have that 
\[
\pi_ {I^i_{(a,A)}}(\Phi_*(a))=\pi_{ I^i_{(a,A)}}(\liftphi(s^i))\text{ for }i=0,1.
\]
 Thus, by shrinking $H$ further if necessary, we can assume that there exists $m\in \bbN$ such that for all $(a,A),(b,B)\in H$, 
\[
\liftphi(a)(n)=\liftphi(s^i)(n) \text{ if }n\geq m \text{ and } n\in I^i_{(a,A)},
\]
and $\liftphi(a)\rs m = \liftphi(b)\rs m$.
 
We now check that for any distinct $(a,A),(b,B)\in H$, and every $ n\in A\cap B$,
\[
\text{ if } a(n)=b(n)\text{ then } \liftphi(a)(n)=\liftphi(b)(n),
\]
which clearly contradicts $K_0$-homogeneity of $H$. Indeed, let $n\in A\cap B$, with $a(n)=b(n)$, if $n< m$ then $\liftphi(a)(n)=\liftphi(b)(n)$, if $n\geq m$, and $n\in I^i_{(a,A)}$, then also $n\in I^i_{(b,B)}$, and 
\[ 
\liftphi(a)(n)=\liftphi(s^i)(n)=\liftphi(b)(n).
\]

Thus far we have proved that there exists no uncountable $K_0$-homo\-ge\-neous subset of $\mathcal{X}$.
By applying $ \OCAT$, let $\mathcal{X}=\bigcup_{n\in\bbN} X_n$, where every $X_n$ is $K_1$-homogeneous. For each $n$ fix a dense countable subset $D_n$ of $X_n$.

Define 
\[
\mathcal{A}_{\text{bad}}=\{ A\in\mathcal{A}: \exists a\in\prod_n\cM_n\,( (a,A)\in\bigcup_n D_n 
)\}.
\]
For every two natural numbers $k$ and $n$, pick a map $h_{kn}\colon\cM_n\to\cN_n$ with the property that for every $u\in\cM_n$: 
\begin{center}\begin{minipage}{0.85\textwidth}
 if there exists $(b,B)\in D_k$ such that $n\in B$ and $b(n)=u$, then there exists $(b,B)\in D_k$ such that $n\in B$, $b(n)=u$ and $\liftphi(b)(n)=h_{kn}(u)$.
\end{minipage} \end{center}
It suffices to prove the following claim to finish the proof of Proposition~\ref{lem:IprodccmodFin}.

\begin{claim}\label{cl:claim prod}
For every $A\in \mathcal{A}\setminus \mathcal{A}_{\textrm{bad}}$ there exists $k\in \bbN$ such that for every $a\in \prod_{n}\cM_n$:
\[
\Phi_*(a)(n)=h_{kn}(a(n)) \quad \forall^\infty n\in A.
\]
\end{claim}

\begin{proof}
 Suppose otherwise, and fix $A\in \mathcal{A}\setminus \mathcal{A}_{bad}$ and a sequence $(a_k)_k$ of elements of $\prod_n\cM_n$ such that for every $k$ the set
\[
Z_k=\{ n\in A: \liftphi(a_k) (n) \neq h_{kn}(a_k(n)) \}
\]
is infinite.

Let $Y_k\subseteq Z_k$ be infinite and pairwise disjoint, and fix $a\in\prod_n\cM_n$ such that $a\rs Y_k=a_k\rs Y_k$ for every $k$. Now suppose $(a,A)\in X_k$. Since $a\rs Y_k=a_k\rs Y_k$, there is $m$ such that for all $n\in Y_k$ with $n\geq m$ we have $\Phi_*(a)(n)=\Phi_*(a_k)(n)$. By density of $D_k$, pick $n\in Y_k$ with $n\geq m$ and $(b,B)\in D_k$ such that 
\[
n\in B\text{ and } b(n)=a(n) \text{ and } \liftphi(b)(n)=\liftphi(a)(n).
\]
By the choice of $h_{kn}$, it follows that there exists $(b',B')\in D_k$ such that $n\in B'$, $b'(n)=a(n)$, and $\liftphi(b')(n)=h_{kn}(a(n))$. By $K_1$-homogeneity of $X_k$, we get $\liftphi(b')(n)=\liftphi(a)(n)$. Thus, 
\[
\liftphi(a_k)(n)=\liftphi(a)(n)=\liftphi(b')(n)=h_{kn}(a(n))=h_{kn}(a_k(n)).
\] 
 This contradicts $n\in Z_k$.
\end{proof}  
This proves that $\cA\setminus\cJprod$ is countable and in particular that $\cJprod$ intersects $\cA$ nontrivially.
\end{proof}

\begin{remark}\label{rmk:noMA}
	All applications of $\MAsl$ in the present paper filter through Proposition~\ref{lem:IprodccmodFin}. We conjecture that the conclusion of Proposition~\ref{lem:IprodccmodFin} follows from $\OCAT$ alone. 
		On the other hand, the proof we have in mind would (if correct) be more laborious, because the key difficulties (first and foremost, the fact that the structures $\cM_n$ and $\cN_n$ can be infinite) seem to demand additional extensive detours to be circumvented.
		It would be desirable to find a concise and elegant proof that $\OCAT$ suffices to prove Proposition~\ref{lem:IprodccmodFin}.
\end{remark}

\subsection{Uniformization}\label{S.Uniformization}

In the remaining part of the proof of Theorem~\ref{T.OCACR} we will not be using $\MAsl$. This is why the conclusion of Proposition~\ref{lem:IprodccmodFin} is included into the assumptions of the main result of this subsection, Proposition~\ref{P.Uniformization}. 

\begin{proposition}\label{P.Uniformization}
Assume $\OCAT$. Let $\cM_n$ and $\cN_n$, for $n\in \bbN$, be countable sets, and let
\[
\Phi\colon \prMF\to \prNF
\]
be a coordinate-respecting function with $\alpha^\Phi = \id$ such that the ideal $\cJprod$ intersects every perfect tree-like almost disjoint family nontrivially. Then~$\Phi$ is of product form.
\end{proposition}

\begin{proof} 
Using the notation set up before Definition~\ref{Def.CR.Th}, if $A\in\mathcal P(\bbN)/\Fin$, then $\Phi_A$ is the induced function between $\prod_n\cM_n/\Fin\restriction A$ and $\prod_{n\in A}\cN_n/\Fin\restriction A$.

For every $A\in \cJprod$ fix a sequence 
\[
g^A=(g^A_n: n\in A)
\] 
of functions with $g^A_n\colon \cM_n\to \cN_n$ such that the function $\Psi_A\colon \prod_{n \in A} \cM_n \to \prod_{n \in A} \cN_n $ defined by
\begin{equation}\label{eq.PsiA}
	\Psi_A(a)(n)=g^A_n(a(n)), \text{ for }n\in A,
\end{equation}
is a lifting of $\Phi_A$.

Consider the Polish space $\cY_n=(\cN_n)^{\cM_n}$. 
Then $\{(g^A,A)\mid A\in \cJprod\}$ is a family of partial selectors as defined and studied in \S\ref{S.partial.selectors}. 
The following shows that the functions $g^A$, for $A \in \cJprod $, form a coherent family in the sense of \cite[\S 2]{Fa:AQ}. 

\begin{claim}\label{C.Coherence} 
For all $A$ and $B$ in $\cJprod $ the set
\[ 
\Diff(g^A,g^B)= \{j \in A \cap B : g^A_j \neq g^B_j\}
\]
is finite. 
\end{claim}

\begin{proof} 
The proof is very similar to the proof of Lemma~\ref{L.hA.coherent}. Fix $A$ and $B$, let $S=\Diff(g^A,g^B)$, and choose $c\in \prod_n\cM_n$ such that $g^A_n(c(n))\neq g^B_n(c(n))$ for all $n\in S$. The assumptions on $\Phi$ imply that $\pi_S(\Phi_A(\pi_A(c)))=\pi_S(\Phi_B(\pi_B(c)))$, but on the other hand $\Psi_A(c)(n)\neq \Psi_B(c)(n)$ for all $n\in S$; therefore $S$ is finite. 
\end{proof}

\begin{claim}\label{C.Uniformization.coherent}
	There are sets $\cG_n$, for $n\in\bbN$, such that $\cJprod=\bigcup_n \cG_n$, and for all $n$ and all $A,B$ in $\cG_n$ we have $|\Diff(g^A,g^B)|\leq n$.
\end{claim}
\begin{proof}	Assume otherwise. By $\OCAT$ and Proposition~\ref{P.Unif} the following holds. 
\begin{enumerate}[label=($\ast$)]
\item\label{c2}There are a perfect tree-like almost disjoint family $\cA$, an uncountable $Z\subseteq \twoN$, and $f\colon Z\to \cJprod$ such that for every $A\in \cA$ and all distinct $x,y$ in $Z$ we have 
\[
(\Diff(g^{f(x)},g^{f(y)})\cap A)\setminus \Delta(x,y)\neq \emptyset. 
\]
\end{enumerate}
 By assumption, $\cJprod$ intersects every perfect tree-like almost disjoint family. Fix $B\in \cA\cap \cJprod$. For $x\in Z$ let $n(x)=\max (\Diff(g^{f(x)},g^B))$ (with $\max\emptyset=0$). Since $Z$ is uncountable, there is $n$ such that $Z'=\{x\in Z\mid n(x) < n\}$ is uncountable. Choose distinct $y,z$ in $Z'$ satisfying $\Delta(y,z)\geq n$ and fix $j$ in $(\Diff(g^{f(y)},g^{f(z)})\cap B)\setminus n$. Then $g^{B}_j=g^{f(y)}_j\neq g^{f(z)}_j=g^B_j$; contradiction. 
\end{proof}

Since $\cJprod$ is nonmeager (again by Talagrand and Jalali-Naini, \cite[Theorem~3.10.1]{Fa:AQ}), there is $\bar n$ such that $\cG_{\bar n}$ as provided by Claim~\ref{C.Uniformization.coherent}
is nonmeager. 
We will attempt to find a sequence $(A(i), B(i), k(i), x(i))$, for $i\in \bbN$, with the following properties for all $i$. 
\begin{enumerate}
	\item $A(i)$ and $B(i)$ belong to $\cG_{\bar n}$. 
	\item $k(i)\in A(i)\cap B(i)$. 
	\item $x(i)\in \cM_{k(i)}$. 
	\item $	g^{A(i)}_{k(i)}(x(i))\neq g^{B(i)}_{k(i)}(x(i))$. 
	\item If $1\leq j<i$ then $g^{A(i)}_{k(j)}(x(j))=g^{A(j)}_{k(j)}(x(j))$ and $g^{B(i)}_{k(j)}(x(j))=g^{B(j)}_{k(j)}(x(j))$. 
	\item Each of the following two sets is nonmeager. 
	\begin{align*}
		\cA_i&=\{C\in \cG_{\bar n}\mid (\forall j<i) k(j)\in C\text{ and } g^C_{k(j)}(x(j))=g^{A(j)}_{k(j)}(x(j))\},\\ 
		\cB_i&=\{C\in \cG_{\bar n}\mid (\forall j<i) k(j)\in C\text{ and } g^C_{k(j)}(x(j))=g^{B(j)}_{k(j)}(x(j))\}. 
	\end{align*}
\end{enumerate}
	These conditions imply that $\Diff(g^{A(i)}, g^{B(i)})\supseteq \{k(j)\mid j\leq i\}$, and by the choice of~$\cG_{\bar n}$, the recursive construction of this sequence has to stop before $i=\bar n$. 
	Note that $\cA_0=\cB_0=\cG_{\bar n}$ and 
	set $k(-1)=-1$. 
	
	Fix the least $i\leq \bar n$ such that $(A(i), B(i), k(i), x(i))$ cannot be chosen (such $i$ exists and $i<\bar n$). What this implies is that for every $k>k(i-1)$, every $x\in \cM_k$, and all $y\neq z$ in $\cN_k$, at least one of the sets $\{C\in \cA_i\mid g^C_k(x)=y\}$ and
	$\{C\in \cB_i\mid g^C_k(x)=z\}$
	is meager. Since $\cM_k$ and $\cN_k$ are countable, this means that for every $k>k(i-1)$ there is a function $\tilde g_k\colon \cM_k\to \cN_k$ such that the set 
	\[
	\calH_k=\{C\in \cA_i \cup \cB_i\mid g^C_k=\tilde g_k\}
	\]	
	is relatively comeager in $\cA_i\cup \cB_i$. 
Therefore $\bigcap_{k>k(i-1)}\calH_k$ is relatively comeager in $\cA_i\cap \cB_i$. 
Let $\hat \calH$ denote the hereditary closure of this set. 

For $A\in \hat \calH$ pick $C\in \calH$ such that $A\subseteq B$ and let $g^A_k=g^C_k$ for all $k\in A\setminus k(i-1)$. Note that $g^A_k=\tilde g_k$ for all $k\in A\setminus k(i-1)$. 
\cite[Theorem~3.10.1(a)]{Fa:AQ} implies that there exists $\bar k$ such that every finite subset of $\bbN\setminus \bar k$ is in $\hat \calH$.

\begin{claim}\label{C.2}
For every $a\in \prod_n \cM_n$ the set $S_a=\{k\mid \Phi_*(a)(k)\neq \tilde g_k(a(k))\}$ is finite. 
\end{claim}

\begin{proof}
Assume otherwise. Since $\hat \cG_{\bar n}$ is nonmeager and hereditary, there is an infinite $X\subseteq S_a$ in $\hat \calH$ (\cite[Theorem~3.10.1 (a)]{Fa:AQ} gives a considerably stronger statement). Let $Y\in\calH$ be such that $X\subseteq Y$. We have $g^Y_k=\tilde g_k$ for all $k>k(i-1)$, in particular for all $k\in X\setminus k(i-1)$. Therefore $g^Y_k(a(k))\neq \Phi_*(a)(k)$ for infinitely many $k$, contradicting our assumption that $g^Y$ witnesses $Y\in \cJprod$. 
\end{proof}

Claim~\ref{C.2} implies that $\tilde g_k$, for $k>k(i-1)$, provide a lifting of product form. This is the thesis.
\end{proof}

\begin{proof}[Proof of Theorem~\ref{T.OCACR}]
Suppose that $\cM_n$ and $\cN_n$, for $n\in \bbN$, are countable sets, and that $\Phi\colon \prMF\to \prNF$
	is a coordinate-respecting function. Exactly as in the discussion at the beginning of \S\ref{S:proof}, we can assume that $\alpha^{\Phi}$, the automorphism of $\mathcal P(\bbN)/\Fin$ given by the definition of coordinate-respecting, is the identity. By $\OCAT$, $\MAsl$, and Proposition~\ref{lem:IprodccmodFin}, the ideal $\cJprod$ intersects every uncountable almost disjoint family, and in particular every perfect tree-like one. By $\OCAT$ and Proposition~\ref{P.Uniformization}, $\Phi$ is of product form. 
\end{proof}
\begin{remark}
In a previous version of this manuscript we had a different, longer, proof of Theorem~\ref{T.OCACR}. We had shown that the hypotheses of Proposition~\ref{P.Uniformization} imply the existence of a Baire-measurable lifting and that every coordinate-respecting map between reduced products with Baire-measurable lifting is of twisted product form. 
\end{remark}
\begin{remark}
In the case when $\sup_n |\cM_n|$ is finite,  a considerably simpler argument shows that every coordinate-respecting map $\Phi\colon\prod \cM_n/\Fin \to \prod\cN_n/\Fin$ such that  $\alpha^\Phi=\id$ is automatically of product form. 
Indeed, suppose that $\sup_n |\cM_n| =m$ is finite  and fix enumerations $\cM_n=\{ x^n_i:i< |\cM_n| \}$ for each $n$. For $i<m$, define $a_i\in\prod_n \cM_n$ by
\[
a_i(n)=\begin{cases}x^n_i&\text{if }i<|\cM_n|,\\
x^n_0&\text{otherwise}.
\end{cases}
\]
	It is easy to verify that the sequence $(g_n)_n$ of maps $g_n\colon \cM_n\to \cN_n$ defined by  $g_n(x^n_i)=\Phi_\ast(a_i)(n)$  satisfies $\Phi[ (g_n) ] = \Phi$.
 In this case Propositions~\ref{lem:IprodccmodFin} and ~\ref{P.Uniformization} are a fortiori true in $\ZFC$, and Theorem~\ref{T.OCACR} follows from (in fact, is equivalent to) the statement that every automorphism of the Boolean algebra $\cP(\bbN)/ \Fin$ is trivial.
 
We suspect that this equivalence does not hold in general, when the structures~$\cM_n$ are allowed to be infinite. 

The assumption that $\alpha^\Phi=\id$ cannot be dropped from the above argument. 
If $\cP(\bbN)/\Fin$ is identified with $\prod_{\Fin} \{0,1\}$ (where $\{0,1\}$ is the two-element Boolean algebra), then each one of its automorphisms is coordinate-respecting, yet $\CH$ implies that some automorphisms are nontrivial, or equivalently, not of product form.
\end{remark}
\section{Applications} \label{S.Applications}

\subsection{Maximal rigidity}
The following is used in the proofs of Theorems~\ref{T.Fields}--\ref{T.Graphs}.

\begin{theorem}\label{cor:triv} 
Assume $ \OCAT+\MAsl$. If $\cM_n$, $\cN_n$, for $n\in \bbN$ are structures of the same finite language whose theory recognizes coordinates, then every isomorphism $\Phi\colon \prMF\to \prNF$ is trivial. 
 
In particular, $\prMF\cong\prNF$ if and only if there is a bijection $f$ between cofinite subsets of $\bbN$ such that $\cM_{f(n)}\cong \cN_n$ for all~$n$. 
\end{theorem}

\begin{proof}
Since the theory of $\cM_n$ and $\cN_n$ recognizes coordinates, $\Phi$ is coordinate-respecting. By Theorem~\ref{th:main}, $\Phi$ is trivial.
\end{proof}

\begin{proof}[Proof of Theorem~\ref{T.Fields}]
Assume $\OCAT+\MAsl$. Suppose that $\cF_n$, $\cG_n$ are two sequences of fields that are finite or countable. 

We need to prove that $\cF=\prod_n \cF_n/\Fin$ and $\cG=\prod_n \cG_n/\Fin$ are isomorphic if and only if there is a bijection $f$ between cofinite subsets of~$\bbN$ such that $\cF_{f(n)}\cong \cG_{n}$ for all $n\in \dom(f)$. Only the direct implication requires a proof, and it follows from Theorem~\ref{cor:triv} and the fact that the theory of fields recognizes coordinates (Lemma~\ref{L.Th.CR.F}). 

This also implies that all automorphisms of $\cF$ (and all isomorphisms between $\cF$ and $\cG$, if any) are trivial. 
\end{proof} 

\begin{proof}[Proof of Theorems~\ref{T.LO} and \ref{T.Trees}]
Assume $ \OCAT+\MAsl$. Recall that linear orders and trees are examples of connected ramified sets (see \S\ref{S.ramified}), and fix an isomorphism between reduced products of connected ramified sets, $\Phi\colon \prod_n \cM_n/\Fin\to \prod_n \cN_n/\Fin$. Since the theory of connected ramified sets recognizes coordinates, (Proposition~\ref{P.ramified.CR}), by Theorem~\ref{cor:triv} $\Phi$ is trivial and there is a bijection $f$ between cofinite subsets of $\bbN$ such that $\cM_{f(n)}\cong \cN_{n}$ for all but finitely many $n$. 
\end{proof}

\begin{proof}[Proof of Theorem~\ref{T.Graphs}] 
Assume $ \OCAT+\MAsl$ and fix $p,q$ in $(0,1)$. 
Recall, from the paragraph preceding the statement of Theorem~\ref{T.Graphs}, that we define $\bbG_{\infty,p}$ to be the random variable $\prod_n \bbG_{n,p}$, where $\bbG_{n,p}$ are independent random graphs, and that 
for every $0<p<1$ this defines Borel probability measure $\mu_p$ on the space of all sequences of graphs~$(\cG_n)_n$ such that $\{0,\dots, n-1\}$ is the set of vertices of~$\cG_n$. We will prove that the set of $(\cG_n)_n$ such that all but finitely many of $\cG_n$ are sufficiently random (Definition~\ref{Def.SuffRG}) has full $\mu_p\times\mu_q$ measure. 

Specializing the definition on \cite[p. 173]{alon2008probabilistic} to the case when $s=3$, we say that a graph~$\cG$ has the full level $3$ extension property if for any set $F$ of three distinct vertices of $\cG$ and any function $\chi\colon F\to \{0,1\}$ there is a vertex $w$ in $\cG\setminus F$ such that for all $v\in F$, $v$ is adjacent to $w$ if and only if $\chi(v)=1$. Let $x_{n,p}$ be the probability that the random graph $\bbG_{n,p}$ does not have the full level $3$ extension property. Then \cite[Theorem~10.7.5]{alon2008probabilistic} implies that 
\[
x_{n,p}\leq \min(3^2 n^3 (1-p^3)^{n-p^3}, 3^2 n^3 (1-(1-p)^3)^{n-(1-p)^3}) . 
\]
If $\max(1-p^3, 1-(1-p)^3)<r<1$, then for all but finitely many $n$ the expression on the right-hand side of the displayed formula is smaller than $r^n$, and therefore $\sum_n x_{n,p}<\infty$. By the Borel--Cantelli lemma, $\cG_n$ has the full level $3$ extension property for all but finitely many $n$ with probability $1$. 

Assume that $(\cG_n)_n$	and $(\mathcal H_n)_n$ are sequences of graphs all but finitely many of which are sufficiently random (Definition~\ref{Def.SuffRG}). Since the set of such sequences has full $\mu_p\times \mu_q$ measure, it remains to prove that any isomorphism 
\[
\Phi\colon \prod_n \cG_n/\Fin\to \prod_n \mathcal H_n/\Fin
\] 
is trivial. 
Since the theory of sufficiently random graphs recognizes coordinates (Proposition~\ref{prop:GRrigid}), by Theorem~\ref{cor:triv} and $ \OCAT+\MAsl$ we have that $\prod_n \cG_n/\Fin\cong \prod_n \mathcal H_n/\Fin$ if and only if $\cG_n\cong \mathcal H_{n}$ for all sufficiently large $n$ (we cannot have $\cG_m\cong \mathcal H_n$ unless $m=n$ because isomorphic graphs have the same number of vertices). 

Since the Rado graph is, being random, sufficiently random, the same argument applies to prove rigidity of its reduced powers; of course in this case an automorphism may involve a permutation of coordinates. 
\end{proof}

It would be desirable to prove the analog of our Theorem~\ref{th:main} with $\Fin$ replaced by other analytic ideals. $\OCAT$ and $\MA$ imply that for many analytic P-ideals (more precisely, for all non-pathological analytic P-ideals) the isomorphism between quotients $\cP(\bbN)/\cI$ exists if and only if the ideals $\cI$ and~$\cJ$ are Rudin--Keisler isomorphic (\cite{Fa:AQ}), and it is relatively consistent with ZFC that quotients over analytic ideals are isomorphic if and only if the ideals are Rudin--Keisler isomorphic (\cite{FaSh:Trivial}, see \cite{Gha:FDD} for a simpler proof). 

This should imply that, assuming $\OCAT$ and $\MA$($\sigma$-linked), if $\cM_n$ and $\cN_n$, for $n\in \bbN$, are models of a theory that recognizes coordinates and $\cI$ and $\cJ$ are nonpathological analytic P-ideals on $\bbN$, then every isomorphism $\Phi\colon \prod_n \cM_n/\cI\to \prod_n \cN_n/\cJ$ is of product form. Our results easily imply the following (extremely) poor man’s version of this conjecture.

\begin{corollary} \label{C.IJ}
Suppose that $\cM_n$ and $\cN_n$ are models of a theory that recognizes coordinates and that $\cI$ and $\cJ$ are ideals such that $\cP(\bbN)/\cI$ and $\cP(\bbN)/\cJ$ are not isomorphic. Then $\prod_n \cM_n/\cI$ and $\prod_n \cN_n/\cJ$ are not isomorphic. 
\end{corollary}

\begin{proof} 
The assumption on $\cI$ and $\cJ$ asserts that $\alpha^\Phi$ as required cannot exist. 
\end{proof} 

\subsection{Continuum Hypothesis and the strong failure of rigidity}\label{S.CH}

We provide a (mostly well-known) proof of Proposition~\ref{P.CH}. 
By \cite{JonOl:Almost}, reduced products over $\Fin$ are countably saturated. Since a nontrivial reduced product over $\Fin$ of countably many structures each of which has cardinality $\leq 2^{\aleph_0}$ has cardinality $2^{\aleph_0}$, CH implies that such reduced products are saturated. Therefore two such reduced products are isomorphic if and only if they are elementarily equivalent. 
The Feferman--Vaught theorem (\cite{feferman1959first}) easily implies that if the theories of $\cM_n$ and $\cN_n$ converge to the same limit in the logic topology, then $\prod_n \cM_n/\Fin$ and $\prod_n \cN_n/\Fin$ are elementarily equivalent (see \cite{ghasemi2014reduced}). 
 Since the logic topology on the space of all theories in a fixed countable language is compact and metrizable, by passing to a subsequence one can assure that the theories converge in the logic topology and therefore \eqref{0.P.CH} and in particular the cases of reduced products of fields, \eqref{1.P.CH}, and of trees, \eqref{2.P.CH}, follow.

\eqref{3.P.CH} 
 Let $\calL_n$ and $\calL'_n$ be sequences of at most countable discrete linear orderings with endpoints such that for every $m\in\bbN$ both sets $\{n\mid |\calL_n|\leq m \}$ and $\{ n\mid |\calL'_n|\leq m \}$ are finite. 
 As in \eqref{1.P.CH} and \eqref{2.P.CH}, it suffices to prove that $\prLF$ and $\prLpF$ are elementarily equivalent. By the Feferman--Vaught theorem it suffices to prove that the theories of every such sequence of linear orderings 
 converge in the logic topology (this implies that they all necessarily converge to the same limit). 
 
 By \L o\'s's Theorem, it suffices to prove that all ultraproducts $\prod_n \calL_n/\cU$ (for $\cU$ nonprincipal) are elementarily equivalent. The assumption on cardinalities of the $\calL_n$'s implies that each such ultrapower has cardinality $2^{\aleph_0}$. 
Consider the equivalence relation $\sim$ on $\prod_n \calL_n/\cU$ defined by $x\sim y$ if and only if the interval $(x,y)$ is finite. Each equivalence class is countable, hence the quotient is a dense linear ordering with endpoints that has cardinality $2^{\aleph_0}$. Moreover, since all $\calL_n$ are discrete, every $\sim$-equivalence class is isomorphic to $\bbZ$, except for the two $\sim$-classes corresponding to either of the endpoints. Since the theory of dense linear orderings with endpoints is complete, the quotient linear orderings are isomorphic. Any isomorphism between them can be lifted to an isomorphism of the ultrapowers because corresponding equivalence classes have the same order type (same as $\bbZ$, $\bbN$, or the converse of $\bbN$). 

\eqref{4.P.CH} 
Recall that $\bbG_{n,p}$ denotes the random graph with~$n$ vertices in which every pair of distinct vertices is connected with probability $p\in (0,1)$, and these events are independent.\footnote{Also recall that this is a misnomer; $\bbG_{n,p}$ is a probability space, not a graph. When we assert that a sequence of random graphs has some property, this means that it has the property with probability 1.} Fix $p$. 
By the 0-1 law for first-order theory of random graphs (\cite{fagin1976probabilities} or \cite{glebskii1969range}) and the Borel--Cantelli lemma, there is a set of full $\mu_p$ measure (see the proof of Theorem~\ref{T.Graphs}), such that for every $(\cG_n)_n$ in this set the theories of graphs $\cG_n$ converge to the same limit. 
Then CH implies that all reduced products of graphs, $\prod_n \cG_n$, for $(\cG_n)_n$ in this set, are isomorphic. This is because, analogously to the proof of the first part of the proof of Theorem~\ref{T.LO}, the
Feferman--Vaught theorem implies that all of these reduced products are elementarily equivalent and under CH countable saturation implies full saturation. 

\section{Rigidity of products}\label{S.products}

Our analysis of coordinate-respecting isomorphisms has quotable consequences that are, to the best of our knowledge, new. As in \S\ref{S.CoordinateRespecting}, one can define trivial isomorphisms between (non-reduced) products. In case of two-element Boolean algebras, the product is naturally isomorphic to $\cP(\bbN)$ and every automorphism of $\cP(\bbN)$ is determined by a permutation of the atoms and therefore trivial. In case of products of fields the triviality is again easily proven, as minimal idempotents are simply characteristic functions on atoms, but the following results may be of independent interest, already known, or both. We leave the proofs to the interested reader.

\begin{proposition} 
If $\calL_n$, $\cM_n$, for $n\in \bbN$, are sequences of connected ramified sets and $\Phi\colon \prod_n \calL_n\to \prod_n \cM_n$ is an isomorphism, then there are a permutation $\pi$ of $\bbN$ and isomorphisms $\sigma_n\colon \calL_n\to \cM_{\pi(n)}$ for $n\in \bbN$ such that, for all $x\in\prod_n\calL_n$, 
\(
\Phi(x)(\pi(n))=\sigma_n(x(n))\).
\end{proposition}

\begin{proposition} 
If $\cG_n$, $\calH_n$, for $n\in \bbN$, are sequences of random graphs and $\Phi\colon \prod_n \cG_n\to \prod_n\calH_n$ is an isomorphism, then there are a permutation $\pi$ of $\bbN$ and isomorphisms $\sigma_n\colon \cG_n\to \calH_{\pi(n)}$ for $n\in \bbN$ such that, for all $x\in\prod_n\mathcal G_n$,
\(
\Phi(x)(\pi(n))=\sigma_n(x(n))\).
\end{proposition}

\section{Concluding remarks}
We have only scratched the surface, and many natural questions remain intact. We list the ones that are in our opinion most interesting.

First, what other theories recognize coordinates in reduced products? Can recognizing coordinates be re-framed in terms of the existence of a copy of $\cP(\bbN)/\Fin$ in the $\eq$ (see \cite[p. 151]{Hodg:Model}) of the reduced product at hand? Or, are theories that recognize coordinates characterized by the existence of local copies of $\cP(\bbN)/\Fin$, like the ones obtained for ramified sets and random graphs (\S\ref{S.local} and \S\ref{S.local2}), in reduced products of its models? An easy observation is that if $T$ recognizes coordinates and $T’$ is a theory in an expanded language extending $T$, then $T’$ recognizes coordinates as well. 

Second, more can be said for certain reduced products of structures whose theory is not necessarily coordinate-recognizing. In this context one could consider more liberal notions of triviality of isomorphisms: ones that allow for the possibility that the structures comprising the product can be further decomposed into products of structures of the same theory. We make the following related observation. Fix a language $\mathcal L$. Let $T$ be an $\mathcal{L}$-theory. We say that products of models of $T$ \emph{have a unique decomposition} if whenever we have natural numbers $n$ and $k$ and models of~$T$ such that $\prod_{i\leq k}\mathcal M_i\cong\prod_{j\leq n}\mathcal N_j$, then $k=n$, and there is a permutation of $n$, $\sigma$, such that $\cM_i\cong\cN_{\sigma(i)}$ for all $i\leq n$. 

\begin{proposition}
Suppose that $T$ recognizes coordinates. Then products of models of $T$ have a unique decomposition.
\end{proposition}

\begin{proof}
If products of models of $T$ do not have a unique decomposition, we can construct an isomorphism between reduced products of models of $T$ which is not coordinate-respecting, exactly as in Example~\ref{Ex.Th.CR}. This contradicts the fact that~$T$ recognizes coordinates.\end{proof}

Given the above, our results provide a roundabout tool for proving that certain classes of objects have unique decompositions. We mainly focused on 
 isomorphisms, and ignored the possibility of obtaining rigidity results for homomorphisms. It should be possible to extend the results of \S\ref {S.Uniformization.Fin} and \S\ref{S:proof} for coordinate-respecting homomorphisms between reduced products, obtaining results along the lines of the $\OCA$ lifting theorem (Theorem~\ref{T.OCAsharp-Fin}). A more delicate issue would be to describe theories which do recognize coordinates for homomorphisms, and whether a local version is sufficient in this case.

In a different direction, focusing on the ideal we are quotienting by, for what reduced products of the form $\prod_n \cM_n/\cI$ can analogous rigidity results be proven? In the case of two-element Boolean algebras and nonpathological analytic P-ideals, an affirmative answer has been provided by the $\OCA$ lifting theorem of \cite{Fa:AQ}. 
Corollary~\ref{C.IJ} is a very soft result in this direction. 
Encouraged by Theorem~\ref{T.OCA.Fin}, we conjecture that $\OCAT$ alone suffices for all rigidity results for quotients of reduced products of countable structures modulo analytic ideals known to follow from $\OCAT$ and $\MA$ (see \cite{farah2022corona} for an overview). 

In the present paper we did not consider reduced products of continuous structures (\cite{BYBHU}). In the well-studied case of \cstar-algebras, satisfactory rigidity results exist even in the wider context of coronas (\cite{mckenney2018forcing}, \cite{vignati2018rigidity}, but see \cite{farah2022corona} for a general picture). It is likely that this can be generalized to other continuous structures.

\appendix 

\section{From $\sigma$-Borel to continuous}\label{S.sigma-Borel}

For the reader’s convenience, following the referee's suggestion, we recall some basic descriptive set-theoretic terminology.

Suppose that $X$ and $Y$ are Polish spaces. A function $\Theta\colon X\to Y$ is called \emph{C-measurable} if it is measurable with respect to the $\sigma$-algebra generated by analytic sets (sometimes denoted $\sigma(\Sigma^1_1)$). 
The following is a consequence of Jankov--von Neumann's uniformization theorem (see \cite[18.A]{Ke:Classical}):

\begin{theorem}\label{JVN}
Let $X$ and $Y$ be Polish spaces, and let $\cZ\subseteq X\times Y$ be analytic. Then there is a C-measurable selection $\Theta$ for $\cZ$; that is, a C-measurable function $\Theta$ whose domain is the projection of $\cZ$ to $X$ and such that $(x,\Theta(x))\in \cZ$ for every $x\in \dom(\Theta)$. \qed 
\end{theorem}

Let $\Phi\colon \cP(\bbN)\to \cP(\bbN)/\Fin$ be a homomorphism. To $\Phi$ one associates the ideals $\cJ_\sigma$ and $\Jcont$ (Definition~\ref{defin:ideals}). The following lemma states that if a set in~$\J_\sigma$ is partitioned into countably many sets, then all but finitely many of them belong to $\Jcont$.

\begin{proposition}\label{P.Jsigma-to-Jcont}
Suppose that $\Phi\colon \cP(\bbN)\to \cP(\bbN)/\Fin$ is a homomorphism, $\bbN=\bigsqcup_n A_n$, and there are Borel-measurable functions $F_n\colon \cP(\bbN)\to \cP(\bbN)$, for $n\in \bbN$, whose graphs cover the graph of a lifting of $\Phi$. Then for all but finitely many $n$ the restriction of $\Phi$ to $\cP(A_n)$ has a continuous lifting. 
\end{proposition}

Our proof of Proposition~\ref{P.Jsigma-to-Jcont} proceeds by a recursive construction that halts at~$n$ such that the restriction of $\Phi$ to $\cP(A_n)$ has a continuous lifting. 
The recursive construction is facilitated by Lemma~\ref{L.Jsigma-to-Jcont} below, and this lemma is preceded by a 
notation-introducing paragraph. 

In the following we identify $\cP(\bbN)$ and $\twoN$, by identifying a subset of $\bbN$ with its characteristic function. If $A=B\sqcup C$ then we naturally identify $\cP(A)$ with $\cP(B)\times \cP(C)$. In this situation, if $\cB\subseteq \cP(B)$ and $\cC\subseteq \cP(C)$ it will then be convenient to write 
\[
\cB\oplus \cC=\{b\cup c\mid b\in \cB, c\in \cC\}. 
\] 
If $s$ is a function from a finite subset of $\bbN$ into $\{0,1\}$ then we write 
\[
[s]=\{x\in \twoN\mid x(i)=s(i)\text{ for all }i\in \dom(s)\}. 
\]

\begin{lemma}\label{L.Jsigma-to-Jcont}
Suppose $\Phi\colon \cP(\bbN)\to \cP(\bbN)/\Fin$ is a homomorphism with lifting $\Phi_*$, $\bbN=A\sqcup B$, $[s] \cap \cP(B)$ is a relatively clopen subset of $\cP(B)$, and $F\colon \cP(\bbN)\to \cP(\bbN)$ is Borel-measurable. 
\begin{enumerate}
\item \label{L.Jsigma-to-Jcont.1} Then the restriction of $\Phi$ to the set $\cT$ of all $a\subseteq A$ such that the set 
\[
\cZ(a)=\{b\in [s]\cap \cP(B)\mid F(a\cup b)\cap \Phi_*(A)=^*\Phi_*(a)\}
\]
is comeager in $[s]\cap \cP(B)$ has a C-measurable lifting.
\item \label{L.Jsigma-to-Jcont.2}If $\cT$ is relatively comeager in some clopen subset of $\cP(A)$, then~$\Phi$ has a continuous lifting on $\cP(A)$. 
\end{enumerate}
\end{lemma}

\begin{proof}
\eqref{L.Jsigma-to-Jcont.1}: For simplicity of notation, we may assume $[s] \cap \cP(B) =\cP(B)$. Since the Boolean operations $\cup$, $\cap$ and $\Delta$ are continuous, the function 
\[
(a,b,c)\mapsto (F(a\cup b)\cap \Phi_*(A))\Delta c
\] 
is Borel, and therefore the set 
\[
\cX=\{(a,b,c)\in \cP(A)\times \cP(B)\times \cP(\bbN)\mid F(a\cup b)\cap \Phi_*(A)=^*c\}
\]
is, being the preimage of $\Fin$ by the nameless Borel function introduced couple of lines ago, Borel. The set
\[
\cY=\{(a,c)\in \cP(A)\times \cP(\bbN)\mid \{b\subseteq B\mid (a,b,c)\in \cX\}\text{ is comeager}\}
\]
is, by Novikov's theorem (\cite[Theorem~29.3]{Ke:Classical}), analytic. By our assumption, for every $a\in \cT$ the set $\cZ(a)=\{b\subseteq B \mid (a,b,\Phi_*(a))\in \cX\}$ is comeager in $\cP(B)$, in particular the section $\cY_a$ is nonempty for all $a\in \cT$. 
 
 Therefore the Jankov--von Neumann uniformization theorem implies that there exists a C-measurable function 
\[
G_0\colon \cP(A)\to \cP(\bbN)
\]
such that for all $a \in \cT$ the set $\cX(a)$ consisting of all $b\subseteq B$ that satisfy
\begin{equation}\label{eq:A3_1}
	F(a\cup b)\cap \Phi_*(A)=^*G_0(a)
\end{equation}	
is comeager. In addition, for every $a \in \cT$ the set $\cZ(a)$ consisting of all $b\subseteq B$ such that
\begin{equation}\label{eq:A3_2}
	F(a\cup b)\cap \Phi_*(A)=^*\Phi_*(a)
\end{equation}
is comeager. 
Hence for each $a \in \cT$ there is $b\in \cX(a)\cap \cZ(a)$, and for this $b$ both (\ref{eq:A3_1}) and (\ref{eq:A3_2}) hold, thus $G_0(a)=^*\Phi_*(a)$.
 
 \eqref{L.Jsigma-to-Jcont.2}: Assume that $\cT$ is relatively comeager in some clopen subset of $\cP(A)$. By the already proven part of this lemma,~$\Phi$ has a C-measurable lifting on $\cT$. Since every C-measurable function is Baire-measurable,~$\Phi$ has a continuous lifting on $\cP(A)$.
\end{proof}

\begin{proof}[Proof of Proposition~\ref{P.Jsigma-to-Jcont}]
Fix a partition $\bbN=\bigsqcup_n A_n$ and Borel-mea\-su\-ra\-ble functions $F_n\colon \cP(\bbN)\to \cP(\bbN)$ whose graphs cover the graph of a lifting $\Phi_*$ of~$\Phi$. 
 
It suffices to prove that the restriction of $\Phi$ to $\cP(A_n)$ has a continuous lifting for some $n$. Assume this is not the case. Since the ideal $\Jcont$ is closed under finite changes of its elements, 
 this implies that for every $n$ and every nonempty clopen subset $[t] \cap \cP(A_n)$ of $\cP(A_n)$, the restriction of $\Phi$ to $[t] \cap \cP(A_n)$ has no continuous lifting. 
 
It will be convenient to write $\textstyle C_n=\bigcup_{j>n}A_j$.

We will recursively choose sets $a_n\subseteq A_n$ and $\cX_n\subseteq \cP(C_n)$, along with clopen subsets $[s_n] \cap \cP(C_n)$ of $\cP(C_n)$ and decreasing sequences $(U_{ni})_i$ of open subsets of $\cP(C_n)$ such that the following conditions hold for all $n$. 
 \begin{enumerate}
 \item\label{I.sigma-Borel} $F_n(\bigcup_{j\leq n} a_j\cup b)\cap \Phi_*(A_n)\neq^* \Phi_*(a_n)$ for all $b\in \cX_n$. 
 \item $\cX_n \supseteq \bigcap_i U_{ni}$ and $U_{ni}$ is an open dense subset of $[s_n]\cap \cP(C_n)$ for every $i \in \bbN$. 
 \item $\{a_{n+1}\}\oplus \cX_{n+1}\subseteq \cX_n$. 
 \item\label{V.sigma-Borel} For every $k < n$, all sets $b \subseteq C_k$ that satisfy both $b \cap \bigcup_{k< j\leq n} A_j = \bigcup_{k < j\leq n} a_j$ and $b \cap C_n \in [s_n]$ belong to $U_{kn}$. 
 \end{enumerate}
We will describe the selection of $a_n$, $\cX_n$, $(U_{ni})_i$ and $[s_n]$.
For $n=0$, our assumption that $\Phi$ has no continuous lifting on $\cP(A_0)$ and Lemma~\ref{L.Jsigma-to-Jcont} together imply that for some $a_0\subseteq A_0$ the set $\cX_0$ of all $x\subseteq C_0$ such that $F_0(a_0\cup x)\cap \Phi_*(A_0)\neq^* \Phi_*(a_0)$ is nonmeager. 
Since this set is, as a preimage of a Borel set by a Borel-measurable function, Borel, there is a clopen set $[s_0] \cap \cP(C_0) \subseteq \cP(C_0)$ such that $[s_0]\cap \cX_0$ is relatively comeager in $[s_0]\cap \cP(C_0)$. Choose a decreasing sequence $(U_{0i})_i$ of dense open subsets of $[s_0]\cap \cP(C_0)$ whose intersection is contained in $[s_0]\cap \cX_0$. 
 
This describes the construction of $a_0$, $\cX_0$, $(U_{0i})_i$ and $[s_0]$. 

Suppose that $a_n$, $\cX_n$, $(U_{ni})_i$ and $[s_n]$ as required had been chosen. Then we can write $[s_n] \cap \cP(C_n)=[t_n] \cap \cP(A_{n+1}) \oplus [u_n] \cap \cP(C_{n+1})$ for clopen sets $[t_n] \cap \cP(A_{n+1}) \subseteq \cP(A_{n+1})$ and $[u_n] \cap \cP(C_{n+1}) \subseteq \cP(C_{n+1})$. By the Kuratowski--Ulam theorem (see e.g., \cite[\S 8.K]{Ke:Classical}), the set 
\begin{multline*}
\cT_n=\{a\in [t_n] \cap \cP(A_{n+1})\mid \text{the set }\{b \in [u_n]\cap \cP(C_{n+1})\mid a\cup b \in \cX_n\}\\ \text{ is relatively comeager in $[u_n]\cap \cP(C_{n+1})$}\}
\end{multline*}
is relatively comeager in $[t_n]\cap \cP(A_{n+1})$. 
Since the intersection of comeager sets is comeager, by the fact that $\Phi$ has no continuous lifting on $[t_n]\cap \cP(A_n)$, and by applying Lemma~\ref{L.Jsigma-to-Jcont} to $F_{n+1}$, $\cT_n$, and $[u_n]$, we can find $a_{n+1} \in [t_n]\cap \cP(A_{n+1})$ such that 
 \begin{multline*}
 \cX_{n+1}=\{b\in [u_n]\cap \cP(C_{n+1})\mid a_{n+1}\cup b\in \cX_n,\\ F_{n+1}(\bigcup_{j\leq n} a_j \cup a_{n+1}\cup b)\cap \Phi_*(A_{n+1})\neq^* \Phi_*(a_{n+1})\}
 \end{multline*}
 is nonmeager in $[u_n]\cap\cP(C_{n+1})$. Being a Borel set, $\cX_{n+1}$ is relatively comeager in $[s_{n+1}]\cap \cP(C_{n+1})$ for some relatively clopen $[s_{n+1}] \cap \cP(C_{n+1})\subseteq [u_n]\cap \cP(C_{n+1})$. We can moreover assume that the clopen set $[s_{n+1}]\cap \cP(C_{n+1})$ was chosen sufficiently small so that for all $k\leq n$ and all $b \in [s_{n+1}]\cap \cP(C_{n+1})$ the set $\bigcup_{k < j\leq n + 1} a_j \cup b$ belongs to $U_{kn}$. Choose then a decreasing sequence $(U_{n+1i})_i$ of dense open subsets of $[s_{n+1}]\cap \cP(C_{n+1})$ with intersection contained in $[s_{n+1}]\cap \cX_{n+1}$.
 
Then $\{a_{n+1}\}\oplus \cX_{n+1}\subseteq \cX_n$, and the sets $a_{n+1}$, $\cX_{n+1}$, $(U_{n+1i})_i$ and $[s_{n+1}]$ satisfy the requirements. 
 
This describes the recursive construction. Given the sets $a_n$, for $n\in \bbN$, let $a=\bigcup_n a_n$. By the assumption, $F_n(a)=^* \Phi_*(a)$ for some $n$. Finally,  \eqref{V.sigma-Borel} implies that $a \cap C_n = \bigcup_{j > n} a_j \in \bigcap_{j} U_{nj} \subseteq \cX_n$. However, $a\cap A_n=a_n$, hence $F_n(a)\cap \Phi_*(A_n)=^* \Phi_*(a_n)$, but this contradicts \eqref{I.sigma-Borel}. 
\end{proof}

\begin{remark}
Proposition~\ref{P.Jsigma-to-Jcont} can be generalized by replacing $\cP(\bbN)$ with an arbitrary product of finite sets (and considering its reduced product over $\Fin$ in place of $\mathcal P(\bbN)/\Fin$). Since we do not need this fact, and its proof is exactly the same as the one of Proposition~\ref{P.Jsigma-to-Jcont} (with more notation), we decided to leave it to the reader.
\end{remark}

\bibliographystyle{amsplain}
\bibliography{Bibliography_Aut}

\providecommand{\bysame}{\leavevmode\hbox to3em{\hrulefill}\thinspace}
\providecommand{\MR}{\relax\ifhmode\unskip\space\fi MR }
\providecommand{\MRhref}[2]{%
  \href{http://www.ams.org/mathscinet-getitem?mr=#1}{#2}
}
\providecommand{\href}[2]{#2}
\begin{thebibliography}{10}

\bibitem{alon2008probabilistic}
N.~Alon and J.~Spencer, \emph{The probabilistic method}, second ed., Wiley.com,
  2000.

\bibitem{BYBHU}
I.~Ben~Yaacov, A.~Berenstein, C.W. Henson, and A.~Usvyatsov, \emph{Model theory
  for metric structures}, Model Theory with Applications to Algebra and
  Analysis, Vol. II (Z.~Chatzidakis et~al., eds.), London Math. Soc. Lecture
  Notes Series, no. 350, London Math. Soc., 2008, pp.~315--427.

\bibitem{benyaacov2024extremal}
I.~Ben~Yaacov, T.~Ibarluc{\'\i}a, and T.~Tsankov, \emph{Extremal models and
  direct integrals in affine logic}, arXiv preprint arXiv:2407.13344 (2024).

\bibitem{burris1979sheaf}
S.~Burris and H.~Werner, \emph{Sheaf constructions and their elementary
  properties}, Transactions of the American Mathematical Society \textbf{248}
  (1979), no.~2, 269--309.

\bibitem{ChaKe}
C.~C. Chang and H.~J. Keisler, \emph{Model theory}, third ed., Studies in Logic
  and the Foundations of Mathematics, vol.~73, North-Holland Publishing Co.,
  Amsterdam, 1990.

\bibitem{debondt2024saturation}
B.~De~Bondt, I.~Farah, and A.~Vignati, \emph{Saturation of reduced products},
  arXiv preprint arXiv:2401.12539 (2024).

\bibitem{dow2014non}
A.~Dow, \emph{A non-trivial copy of {$\beta N \backslash N$}}, Proc. Amer.
  Math. Soc. \textbf{142} (2014), no.~8, 2907--2913.

\bibitem{fagin1976probabilities}
R.~Fagin, \emph{Probabilities on finite models}, The Journal of Symbolic Logic
  \textbf{41} (1976), no.~1, 50--58.

\bibitem{Fa:Cauchy}
I.~Farah, \emph{Cauchy nets and open colorings}, Publ. Inst. Math. (Beograd)
  (N.S.) \textbf{64(78)} (1998), 146--152.

\bibitem{Fa:AQ}
\bysame, \emph{Analytic quotients: theory of liftings for quotients over
  analytic ideals on the integers}, Mem. Amer. Math. Soc. \textbf{148} (2000),
  no.~702, xvi+177.

\bibitem{Fa:Liftings}
\bysame, \emph{Liftings of homomorphisms between quotient structures and {U}lam
  stability}, Logic Colloquium '98, Lecture notes in logic, vol.~13, A.K.
  Peters, 2000, pp.~173--196.

\bibitem{Fa.Luzin}
\bysame, \emph{Luzin gaps}, Trans. Amer. Math. Soc. \textbf{356} (2004), no.~6,
  2197--2239.

\bibitem{Fa:All}
\bysame, \emph{All automorphisms of the {C}alkin algebra are inner}, Ann. of
  Math. (2) \textbf{173} (2011), 619--661.

\bibitem{Fa:STCstar}
\bysame, \emph{Combinatorial set theory of \cstar-algebras}, Springer
  Monographs in Mathematics, Springer, 2019.

\bibitem{farah2022corona}
I.~Farah, S.~Ghasemi, A.~Vaccaro, and A.~Vignati, \emph{Corona rigidity}, arXiv
  preprint arXiv:2201.11618 (2022).

\bibitem{FaSh:Trivial}
I.~Farah and S.~Shelah, \emph{Trivial automorphisms}, Israel J. Math.
  \textbf{201} (2014), 701--728.

\bibitem{feferman1959first}
S.~Feferman and R.~Vaught, \emph{The first order properties of products of
  algebraic systems}, Fundamenta Mathematicae \textbf{47} (1959), no.~1,
  57--103.

\bibitem{frieze2016introduction}
A.~Frieze and M.~Karo{\'n}ski, \emph{Introduction to random graphs}, Cambridge
  University Press, 2016.

\bibitem{Gha:FDD}
S.~Ghasemi, \emph{Isomorphisms of quotients of {FDD}-algebras}, Israel J. Math.
  \textbf{209} (2015), no.~2, 825--854.

\bibitem{ghasemi2014reduced}
\bysame, \emph{Reduced products of metric structures: a metric
  {F}eferman--{V}aught theorem}, J. Symbolic Logic \textbf{81} (2016), no.~3,
  856--875.

\bibitem{GilbertGraphs}
E.~N. Gilbert, \emph{Random graphs}, Ann. Math. Statist. \textbf{30} (1959),
  1141--1144. \MR{108839}

\bibitem{glebskii1969range}
Yu.V. Glebskii, D.I. Kogan, M.I. Liogon'kiI, and V.A. Talanov, \emph{Range and
  degree of realizability of formulas in the restricted predicate calculus},
  Cybernetics \textbf{5} (1969), no.~2, 142--154.

\bibitem{Hau:Summen}
F.~Hausdorff, \emph{Summen von $\aleph_1$ {M}engen}, Fund. Math. \textbf{26}
  (1936), 241--255.

\bibitem{Hodg:Model}
W.~Hodges, \emph{Model theory}, Encyclopedia of Mathematics and its
  Applications, vol.~42, Cambridge university press, 1993.

\bibitem{JonOl:Almost}
B.~Jonsson and P.~Olin, \emph{Almost direct products and saturation},
  Compositio Math. \textbf{20} (1968), 125--132.

\bibitem{Just:WAT}
W.~Just, \emph{A weak version of {AT} from {OCA}}, MSRI Publications
  \textbf{26} (1992), 281--291.

\bibitem{KanRe:Ulam}
V.~Kanovei and M.~Reeken, \emph{On {U}lam's problem concerning the stability of
  approximate homomorphisms}, Tr. Mat. Inst. Steklova \textbf{231} (2000),
  249--283.

\bibitem{Ke:Classical}
A.S. Kechris, \emph{Classical descriptive set theory}, Graduate Texts in
  Mathematics, vol. 156, Springer, 1995.

\bibitem{Ku:Book}
K.~Kunen, \emph{Set theory: An introduction to independence proofs},
  North--Holland, 1980.

\bibitem{Kure:Ensembles}
{\DJ}.~Kurepa, \emph{Ensembles ordonn\'{e}es et ramifi\'{e}s}, Publ. Math.
  Univ. Belgrade \textbf{4} (1935), 1--138.

\bibitem{mckenney2018forcing}
P.~McKenney and A.~Vignati, \emph{Forcing axioms and coronas of
  {$\mathrm{C}^*$}-algebras}, J. Math. Log. \textbf{21} (2021), no.~2, Paper
  No. 2150006, 73.

\bibitem{moore2021some}
J.T. Moore, \emph{Some remarks on the open coloring axiom}, Annals of Pure and
  Applied Logic \textbf{172} (2021), no.~5, 102912.

\bibitem{Ru}
W.~Rudin, \emph{Homogeneity problems in the theory of \v{C}ech
  compactifications}, Duke Mathematics Journal \textbf{23} (1956), 409--419.

\bibitem{Sh:PIF}
S.~Shelah, \emph{Proper and improper forcing}, Perspectives in Mathematical
  Logic, Springer, 1998.

\bibitem{shelah1988zero}
S.~Shelah and J.~Spencer, \emph{Zero-one laws for sparse random graphs},
  Journal of the American Mathematical Society \textbf{1} (1988), no.~1,
  97--115.

\bibitem{ShSte:PFA}
S.~Shelah and J.~Stepr{\=a}ns, \emph{{PFA} implies all automorphisms are
  trivial}, Proceedings of the American Mathematical Society \textbf{104}
  (1988), 1220--1225.

\bibitem{ShSte:Non-trivial}
\bysame, \emph{Non-trivial homeomorphisms of {$\beta{\mathbb N}\setminus
  {\mathbb N}$} without the continuum hypothesis}, Fundamenta Mathematicae
  \textbf{132} (1989), 135--141.

\bibitem{vaccaro2022trivial}
A.~Vaccaro, \emph{Trivial endomorphisms of the {C}alkin algebra}, Israel
  Journal of Mathematics \textbf{247} (2022), no.~2, 873--903.

\bibitem{Ve:OCA}
B.~Veli\v{c}kovi\'{c}, \emph{{OCA} and automorphisms of {${\mathcal P}(\omega)
  /\Fin$}}, Top. Appl. \textbf{49} (1993), 1--13.

\bibitem{vignati2018rigidity}
A.~Vignati, \emph{Rigidity conjectures for continuous quotients}, Ann. Sci.
  \'{E}c. Norm. Sup\'{e}r. (4) \textbf{55} (2022), no.~6, 1687--1738.

\end{thebibliography}
\end{document}